\documentclass[11pt,oneside,reqno]{amsart}
\numberwithin{equation}{section}

\usepackage[english]{babel}
\usepackage[T1]{fontenc}
\usepackage[utf8]{inputenc}
\usepackage{amsmath}
\usepackage{amssymb}
\usepackage{amsthm}
\usepackage{color}
\usepackage[a4paper]{geometry}
\usepackage{tikz}
\usepackage{mathdots}
\usepackage[all]{xy} 
\usepackage{hyperref}
\usepackage{accents}
\usepackage{lmodern}
\usepackage{mathrsfs}

\theoremstyle{definition}
\newtheorem{definition}{Definition}[section]

\newtheorem{rmk}[definition]{Remark}

\theoremstyle{plain}
\newtheorem{theorem}[definition]{Theorem}
\newtheorem{prop}[definition]{Proposition}
\newtheorem{cor}[definition]{Corollary}
\newtheorem{lem}[definition]{Lemma}

\newcommand{\mb}{\mathbb}
\newcommand{\mc}{\mathcal}

\newcommand{\SL}{\operatorname{SL}}

\newcommand{\SO}{\operatorname{SO}}

\newcommand{\cF}{\mathcal{F}}

\newcommand{\cP}{\mathcal{P}}

\newcommand{\bN}{\mathbb{N}}

\newcommand{\bR}{\mathbb{R}}

\newcommand{\bZ}{\mathbb{Z}}

\newcommand{\at}{a}
\newcommand{\bt}{b}
\newcommand{\ct}{c}

\newcommand{\ra}{\rightarrow}

\newcommand{\qand}{\quad \textrm{and} \quad}

\newcommand\subsetsim{\mathrel{%
\ooalign{\raise0.2ex\hbox{$\subset$}\cr\hidewidth\raise-0.8ex\hbox{\scalebox{0.9}{$\sim$}}\hidewidth\cr}}}
\newcommand{\eps}{\varepsilon}

\DeclareMathOperator{\supp}{supp}

\DeclareMathOperator{\height}{ht}

\DeclareMathOperator{\Vol}{Vol}

\newcommand{\ul}[1]{\underline{#1}}
\newcommand{\tildea}{\tilde A}

\newcommand{\tildeo}{\tilde \Omega}

\DeclareMathSymbol{\shortminus}{\mathbin}{AMSa}{"39}

\usepackage{scalerel,stackengine}
\stackMath
\newcommand\reallywidehat[1]{%
\savestack{\tmpbox}{\stretchto{%
  \scaleto{%
    \scalerel*[\widthof{\ensuremath{#1}}]{\kern-.6pt\bigwedge\kern-.6pt}%
    {\rule[-\textheight/2]{1ex}{\textheight}}
  }{\textheight}%
}{0.5ex}}%
\stackon[1pt]{#1}{\tmpbox}%
}
\begin{document}

\title[CLTs in multiplicative diophantine approximation]{Central Limit Theorems in Multiplicative Diophantine Approximation}


\author[M. Bj\"orklund]{Michael Bj\"orklund}
\address{Department of Mathematics, Chalmers, Gothenburg, Sweden}
\email{micbjo@chalmers.se}

\author[R. Fregoli]{Reynold Fregoli}
\address{Department of Mathematics, University of Michigan, U.S.}
\email{fregoli@umich.edu}

\author[A.  Gorodnik]{Alexander Gorodnik}
\address{Department of Mathematics, University of Z\"urich, Switzerland}
\email{alexander.gorodnik@math.uzh.ch}

\subjclass[2020]{Primary: 11K60, Secondary: 37A25, 37A44}
\keywords{Multiplicative Diophantine approximation, lattice point counting}

\begin{abstract}
We investigate the number of integer solutions to a multiplicative Diophantine approximation problem and show that the associated counting function converges in distribution to a normal law. Our approach relies on the analysis of correlations of measures on homogeneous spaces, together with estimates for Siegel transforms restricted to subspaces.
\end{abstract}

\maketitle

\section{Introduction}

For a pair of real numbers $\underline{x}=(x_1,x_2)$, we consider the products
\[
q\,\|qx_1\|\,\|qx_2\|,\qquad q\in\mathbb{N},
\]
where $\|\cdot\|$ denotes the distance to the nearest integer. Our interest lies in small values of these products. Accordingly, for $b\in(0,1)$ and $T>0$, we introduce the counting function
\[
N_{b,T}(\underline{x})
:=\big|\{q\in\mathbb{N}:\ q\,\|qx_1\|\,\|qx_2\|<b,\ 1\le q<T\}\big|.
\]
Throughout the paper, both $b$ and $T$ are treated as varying parameters. For example, in Theorem~\ref{th:main1} below one may take $b=T^{-\eta}$ with $\eta<6/5$.

The asymptotic behavior of this and related counting functions has been studied in
\cite{WY,Wid,BFG}. It is convenient to interpret this problem in terms of lattice
point counting. Indeed, consider the domains
\begin{equation}\label{eq:oomega}
\Omega_{b,T}:=\Big\{(x_1,x_2,y)\in\mathbb{R}^3:\ 
|x_1|\,|x_2|\,y<b,\ 
|x_1|,|x_2|\le \tfrac12,\ 
1\le y<T
\Big\}.
\end{equation}
When $\underline{x}$ has irrational coordinates, $N_{b,T}(\underline{x})$ coincides
with the number of integer triples $(p_1,p_2,q)\in\mathbb{Z}^3$ such that
\[
(qx_1+p_1,qx_2+p_2,q)\in\Omega_{b,T}.
\]
It was shown in \cite{BFG} that, under suitable conditions on the parameters $b$ and
$T$, for almost every $\underline{x}$,
\[
N_{b,T}(\underline{x})\sim \Vol(\Omega_{b,T})
\qquad\text{as }T\to\infty.
\]

The main objective of the present paper is to study the distributional behavior of
the functions $N_{b,T}$. In particular, we establish the following central limit theorem.

\begin{theorem}\label{th:main1}
Assume that
\[
(\ln T)^{-6/5+\theta}\le b<1
\qquad\text{for some }\theta>0.
\]
Then the normalized functions
\[
D_{b,T}(\underline{x})
:=\frac{N_{b,T}(\underline{x})-\Vol(\Omega_{b,T})}
{\Vol(\Omega_{b,T})^{1/2}},\qquad 
\underline{x}\in[0,1)^2,
\]
converge in distribution, as $T\to\infty$, to a non-degenerate normal law. More precisely, for every
$\xi\in\mathbb{R}$,
\[
\Vol\Big(\big\{\underline{x}\in[0,1)^2:\ D_{b,T}(\underline{x})<\xi\big\}\Big)
\longrightarrow
\frac{1}{\sqrt{2\pi}\sigma}
\int_{-\infty}^{\xi} e^{-t^2/(2\sigma^2)}\,dt,
\]
where
\[
\sigma^2
:=
\sum_{p,q\in\mathbb{N}}
\Vol\big(p^{-1}\cdot \Delta_\infty \cap q^{-1}\cdot \Delta\big) > 0,
\]
with
\begin{align*}
\Delta_\infty
&:=\big\{(x_1,x_2,y)\in\mathbb{R}^3:\ 
0<|x_1x_2|\,y\le 1\big\},\\
\Delta
&:=\big\{(x_1,x_2,y)\in\mathbb{R}^3:\ 
0<|x_1x_2|\,y\le 1,\ 
(2e)^{-1}<|x_1|,|x_2|\le \tfrac12
\big\}.
\end{align*}
\end{theorem}

\vspace{0.5cm}

More generally, fix $c\in(0,1)$. For parameters $0\le a<b<1$ and $T>1$, we consider
the counting function
\begin{equation}\label{eq:NNN}
N_{a,b,T}(\underline{x})
:=\left|\left\{(p_1,p_2,q)\in\mathbb{Z}^3:\ 
\begin{array}{l}
a<|qx_1-p_1|\,|qx_2-p_2|\,q<b,\\[2pt]
|qx_1-p_1|,\,|qx_2-p_2|\le c,\quad 1\le q<T
\end{array}
\right\}\right|.
\end{equation}
Associated to this counting function is the family of domains
\begin{equation}\label{eq:omega0}
\Omega_{a,b,T}
:=\left\{(x_1,x_2,y)\in\mathbb{R}^{3}:\ 
a<|x_1|\,|x_2|\,y<b,\ 
|x_1|,|x_2|\le c,\ 
1\le y<T
\right\}.
\end{equation}

We establish the following generalization of Theorem~\ref{th:main1}.

\begin{theorem}\label{th:main2}
Assume that
\[
a \le b\,(\ln\ln T)^{-4-\theta}
\qquad\text{and}\qquad
(\ln T)^{-6/5+\theta}\le b<1
\]
for some $\theta>0$. Then the normalized functions
\[
D_{a,b,T}(\underline{x})
:=\frac{N_{a,b,T}(\underline{x})-\Vol(\Omega_{a,b,T})}
{\Vol(\Omega_{a,b,T})^{1/2}},
\qquad \underline{x}\in[0,1)^2,
\]
converge in distribution, as $T\to\infty$, to a normal law with variance
\[
\sigma_c^2
:=
\sum_{p,q\in\mathbb{N}}
\Vol\big(p^{-1}\cdot \Delta_\infty \cap q^{-1}\cdot \Delta_c\big),
\]
where
\begin{equation}\label{eq:Deltac}
\Delta_c
:=\left\{(x_1,x_2,y)\in\mathbb{R}^3:\ 
0<|x_1x_2|\,y\le 1,\ 
e^{-1}c<|x_1|,|x_2|\le c
\right\}.
\end{equation}
\end{theorem}

In the one-dimensional setting, the distribution of the counting function for the
number of integer solutions $(p,q)$ to the inequalities
\[
|qx-p|<\psi(q),\qquad 1\le q<T,
\]
was studied in \cite{lev1,lev2,ph,f} using continued fraction techniques, where a
central limit theorem was established.

In higher dimensions, a central limit theorem was proved in
\cite{dfv,BG19} for the number of integer solutions $(p_1,\ldots,p_n,q)$
to the system of inequalities
\[
|qx_1-p_1|<b_1 q^{-w_1},\ \ldots,\ |qx_n-p_n|<b_n q^{-w_n},\qquad 1\le q<T,
\]
with fixed constants $b_i$ and exponents satisfying
$w_1+\cdots+w_n=1$.

While the present paper follows the general approach developed in
\cite{BG19}, the multiplicative setting considered here presents several
new challenges, which we outline in the next section. A related
multiplicative Diophantine approximation problem involving products of
linear forms was studied in our earlier work \cite{BG}. That work relied
on a different tessellation construction and the method was effective
only in dimensions greater than nine.

\section{Outline of the Proof Strategy}

Our approach relies on analysis on the space $X$ of unimodular lattices in $\mathbb{R}^3$. 
The connection with the original counting problem is provided by the Siegel transform. 
For a compactly supported function $f\colon\mathbb{R}^3\to\mathbb{R}$, its Siegel transform 
$\widehat f\colon X\to\mathbb{R}$ is defined by
\[
\widehat f(\Lambda)
:=\sum_{\underline{\lambda}\in\Lambda\setminus\{0\}} f(\underline{\lambda}),
\qquad \Lambda\in X.
\]

To simplify notation, we suppress the dependence on the parameters $a,b,T$ and write 
$\Omega:=\Omega_{a,b,T}$ for the domain defined in \eqref{eq:oomega}. 
The counting function \eqref{eq:NNN} can then be interpreted as a function on $X$:
\[
N_{a,b,T}(\underline{x})
=|\Lambda_{\underline{x}}\cap\Omega|
=\widehat{\chi}_{\Omega}(\Lambda_{\underline{x}}),
\]
where
\[
\Lambda_{\underline{x}}
:=\big\{(qx_1+p_1,qx_2+p_2,q):\ p_1,p_2,q\in\mathbb{Z}\big\}\in X.
\]
To prove our main results, we analyze the moments of $N_{a,b,T}$.

In Section~\ref{sec:tess}, we introduce a tessellation of $\Omega$ of the form
\[
\Omega=\bigsqcup_{n\in\mathcal{F}_\Omega} a(n)^{-1}\Delta_{\Omega,n},
\]
constructed using the diagonal transformations
\[
a(t):=\operatorname{diag}\big(e^{t_1},e^{t_2},e^{-t_1-t_2}\big),
\qquad t=(t_1,t_2)\in\mathbb{R}_+^2.
\]
This decomposition allows us to write
\[
\widehat{\chi}_{\Omega}(\Lambda_{\underline{x}})
=\sum_{n\in\mathcal{F}_\Omega}
\widehat{\chi}_{\Delta_{\Omega,n}}\big(a(n)\Lambda_{\underline{x}}\big),
\]
and hence to express the moments as
\[
\int_{[0,1)^2} N_{a,b,T}(\underline{x})^r\,d\underline{x}
=\sum_{(n_1,\ldots,n_r)\in\mathcal{F}_\Omega^r}
\int_{[0,1)^2}
\prod_{i=1}^r
\widehat{\chi}_{\Delta_{\Omega,n_i}}\big(a(n_i)\Lambda_{\underline{x}}\big)
\,d\underline{x}.
\]

In practice, we work with smooth compactly supported approximations of the functions 
$\widehat{\chi}_{\Delta_{\Omega,n_i}}$, constructed in Section~\ref{sec:appccs}. 
The above formula shows that the analysis of moments of $N_{a,b,T}$ reduces, 
to a large extent, to understanding the convergence of certain measures on the product space $X^r$. 
Let $\nu$ denote the measure on $X$ supported on the family of lattices $\Lambda_{\underline{x}}$, 
$\underline{x}\in[0,1)^2$, and let $\nu_r^\Delta$ be its image under the diagonal embedding into $X^r$. 
The problem then becomes the study of the push-forward measures
\[
(a(t_1),\ldots,a(t_r))_*\nu_r^\Delta,
\qquad t_1,\ldots,t_r\in\mathbb{R}_+^2,
\]
on $X^r$.

This question was investigated in our companion paper \cite{BFG2}, and the relevant results are summarized in Section~\ref{sec:corr}. 
In contrast to the one-parameter setting considered in \cite{BG19}, this family of measures 
admits many accumulation points supported on proper homogeneous subspaces of $X^r$ (see Theorem~\ref{prop:ED}). 
Excluding the basins of attraction of these limiting measures—where correlations can be analyzed using Theorem~\ref{prop:ED}—still leaves a large number of terms whose total contribution is not $o\big(\Vol(\Omega)^{1/2}\big)$.

A key new ingredient is a decorrelation estimate (Theorem~\ref{prop:integrability}), 
in which the error term depends only on the mutual distances $|t_i-t_j|$. 
This estimate plays a crucial role in the analysis of correlation sums and ultimately allows us 
to exploit cancellations in cumulants, making it possible to control higher-order moments.

The presence of limiting measures supported on proper homogeneous subspaces leads to an additional difficulty: 
to obtain effective bounds, we also need estimates for the $L^p$-norms of Siegel transforms restricted to these subspaces. 
This analysis is carried out in Section~\ref{sec:appest}. 
The resulting estimates are independent of the rest of the paper and may be of interest for other problems in multiplicative Diophantine approximation.

Using these Siegel transform bounds, we analyze the variance in Sections~\ref{sec:var1} and~\ref{sec:CompVar}. 
In Section~\ref{sec:var1}, we treat the contribution coming from basins of attraction of proper limiting measures, 
while Section~\ref{sec:CompVar} deals with the remaining terms. 
Higher-order cumulants are handled in Section~\ref{sec:higherord}. 
Finally, the proof of the main theorem is completed in Section~\ref{sec:compofp}. 
Combining the variance and cumulant estimates obtained in these sections, 
we deduce convergence to the normal law under suitable assumptions on the parameters.

Our method requires an additional condition of the form $a\ge(\ln T)^{-K}$ for some $K>0$. 
The complementary case is handled by an application of a Borel–Cantelli argument (see Lemma~\ref{l:small}).


\section{Tessellations} \label{sec:tess}

Throughout the paper, we work with the sets
$\Omega:=\Omega_{a,b,T}$ introduced in \eqref{eq:omega0}
depending on $0<a<b<1$, $0<c<1$, and $T> 1$. 
While $c$ is fixed through the paper, $a,b,T$ are viewed as varying parameter. A direct computation gives that
$$
V_\Omega:=\Vol(\Omega) = 2(\ln T)^2 (b-a) + 4(\ln T) \left( (b-a) + b \ln\left(b^{-1}c^2\right) - a \ln\left(a^{-1}c^2\right) \right) + O_c(1).
$$
In particular, it follows that 
\begin{equation}\label{eq;voll}
V_\Omega:=\Vol(\Omega)\sim 2(\ln T)^2(b-a)\quad\hbox{as $T\to\infty$}
\end{equation}
when
$$
a\gg (\ln T)^{-K}\;\;\hbox{for some $K>0$}\qand (\ln T)^2(b-a)\to\infty.
$$

A crucial ingredient of our argument is the tessellation constructed in our previous work \cite[Lemma 3.1]{BFG}:
\begin{equation}\label{eq:tess}
\Omega =  \bigsqcup_{n \in \mathcal{F}_\Omega} a({n})^{-1}\Delta_{\Omega,{n}},
\end{equation}
where
\begin{align}
\label{Def_DeltaTn}
\Delta_{\Omega,{n}} &:= \left\{ (x_1,x_2,y)\in\bR^3: \;
\begin{tabular}{l}
 $\at < |x_1 x_2|\, y \leq \bt,\;
e^{-1} \ct  < |x_1|,\,|x_2| \leq \ct$\\
 $e^{-n_1 - n_2} \leq y < T e^{-n_1 - n_2}$
\end{tabular}
\right\}, \\
\label{Def_FT}
\cF_\Omega &:= \left\{ n \in \mb N_o^2 \,  : \,  \alpha_\Omega 
\leq n_1 + n_2 < \beta_\Omega \right\}
\end{align}
with
\begin{equation}
\label{Def_AB}
\alpha_\Omega := \ln\left(e^{-2} \bt^{-1} \ct^2 \right) \qand \beta_\Omega :=  \ln\left(T\at^{-1}\ct^2\right).
\end{equation}
We have 
\begin{equation}\label{eq:cardQRT}
|\cF_\Omega|\le (\beta_\Omega+1)(\beta_\Omega-\alpha_\Omega+1)\ll_c (\ln T)^2.
\end{equation}
One can check (see \cite[Lemma 3.1]{BFG}) that
\begin{equation}\label{eq:subb}
\Delta_{\Omega,{n}} 
\subset
 \left[- \ct,\ct\right]^2
 \times 
 \left({\at}{\ct^{-2}},e^2 {\bt }{\ct^{-2}}\right]\quad\hbox{for all $n$.
}
\end{equation}
We use a certain regularity property of the sets $\Delta_{\Omega,{n}}$, which is the content of 
\cite[Lemma~4.3]{BFG}. For $\eps > 0$, let $V_\eps$ denote 
the symmetric open neighborhood of the identity in $\SL_3(\bR)$
defined by
\begin{equation}\label{eq:ve}
V_\varepsilon:=\{g\in \SL_3(\bR):\|g-\textup{id}\|_{\textup{op}},\|g^{-1}-\textup{id}\|_{\textup{op}}<\varepsilon\},
\end{equation}
where $\|\cdot\|_{\textup{op}}$ stands for the operator norm of a matrix induced by the supremum norm on $\mb R^3$. 
For $E \subset \bR^3$ and $y\in\bR$, we write
\[
E^y := \{ (x_1,x_2) \in \bR^2 \,  : \,  (x_1,x_2,y) \in E \}.
\]
We are interested in comparing the sets $\Delta_{\Omega,{n}}$ and $g^{-1}\Delta_{\Omega,{n}}$ with $g\in V_\eps$.
In view of \eqref{eq:subb}, the sets $\Delta_{\Omega,{n}}$ satisfy the conditions of 
\cite[Lemma 4.3]{BFG} with $M=e^2 c^{-2}$ and  $\gamma=a c^{-2}$.
Therefore, we deduce from this lemma that for every $\eps\in (0, O_c(a))$, there exist 
sets $E_s\subset \bR^2 \times \bR$, $s=1,\ldots,24$, such that
\begin{equation}\label{eq:controled}
\left(g^{-1}\Delta_{\Omega,{n}} \setminus \Delta_{\Omega,{n}}\right) \sqcup \left(\Delta_{\Omega,{n}} \setminus g^{-1}\Delta_{\Omega,{n}}\right) \subset \bigcup_{s} E_s\quad\hbox{for all $g \in V_\eps$, }
\end{equation}
where each of these sets satisfies 
\begin{equation}\label{eq:controled1}
E_s \subset \left[-(e^2 c^{-2}+1),e^2 c^{-2}+1\right]^2 \times \left(a c^{-2}/2,e^2 c^{-2}+1\right]
\end{equation}
and
\begin{equation}\label{eq:controled2}
\Vol(E_s^y)
\ll_c
\max\left(\eps,-\frac{\eps}{a} \ln\left(\frac{\eps}{a}\right)\right)\quad\hbox{for all $y$,}
\end{equation}
 or  there is an interval $[\alpha,\beta]$ such that
\begin{equation}\label{eq:controled3}
E_s \subset \left[-(e^2 c^{-2}+1),e^2 c^{-2}+1\right]^2 \times [\alpha,\beta],  \quad \beta-\alpha \ll_c \eps,\quad \alpha\geq a c^{-2}/4.
\end{equation}
In particular, it follows from \eqref{eq:controled2}--\eqref{eq:controled3} that 
\begin{equation}\label{eq:es}
\Vol(E_s)
\ll_c
\max\left(\eps,-\frac{\eps}{a} \ln\left(\frac{\eps}{a}\right)\right).
\end{equation}
We also note that the sets $E_s$ constructed in \cite[Lemma 4.3]{BFG} are given by finitely many linear inequalities.

\section{Correlations on homogeneous spaces}\label{sec:corr}
Our analysis is carried out on the space of unimodular three-dimensional lattices
$$
X\simeq G/\Gamma,\quad\hbox{where $G:=\hbox{SL}_3(\bR)$ and $\Gamma:=\hbox{SL}_3(\bZ)$,}
$$
equipped with the probability invariant measure $\mu$.
We consider the homogeneous subspace
\begin{equation}\label{eq:u}
Y:=U\Gamma\subset X,\quad\hbox{where $U:=\begin{pmatrix} 1 & 0 & * \\
0 & 1 & * \\
0 & 0 & 1\end{pmatrix}$},
\end{equation}
equipped with the probability invariant measure $\nu$. 
The crucial question that we will need to solve for the present work and which was explored in details in our companion paper \cite{BFG2}
is the limiting behavior of the translated measures $a(t)_*\nu$ with $t\in \bR_+^2$ and their correlations. 
In particular, we use the following decorrelation estimate for these measures, where $|\cdot |$ denotes
the $\max$-norm on $\bR^2$ and $\|\cdot \|_{C^\ell}$ stands for the $C^\ell$-norm on the space of smooth functions on $X$.

\begin{theorem}[\cite{BFG2},Theorem 1.4]
\label{prop:integrability}
Let $r\in\mb{N}$. There exist $\ell\ge 1$ and $\eta_r>0$ such that for any 
$t_{1},\dotsc,t_{r}\in\bR^{2}_+$ and 
$\varphi_{1},\dotsc,\varphi_{r}\in C_c^\infty(X)$,
$$
\left|\int_{Y}\prod_{i=1}^{r}\varphi_{i}\circ a(t_{i})\,d\nu-\prod_{i=1}^{r}\int_{Y}\varphi_{i}\circ a(t_{i})\,d\nu\right|\ll_{r}e^{-\eta_{r}\min_{i\neq j}|t_{i}-t_{j}|}
\prod_{i=1}^{r}\|\varphi_{i}\|_{C^\ell}.
$$
\end{theorem}

Depending on the parameter $t$, the measures $a(t)_*\nu$ converge to measures supported on different homogeneous subspaces.  The following homogeneous subspaces come into play:
\begin{equation}\label{eq:yy}
Y_1:=G_1\Gamma \;\;\hbox{and}\;\; Y_2:=G_2\Gamma \quad\hbox{ with 
$G_{1}:=\begin{pmatrix} * & 0 & * \\
* & 1 & * \\
* & 0 & *\end{pmatrix}$
\;\;\hbox{and}\;\; $G_{2}:=\begin{pmatrix} 1 & * & * \\
0 & * & * \\
0 & * & *\end{pmatrix}
$
}.
\end{equation}
We denote by $\nu_1$ and $\nu_2$  the invariant probability measures supported on $Y_1$ and $Y_2$ respectively.
One can check that 
$$
a(t_1,0)_*\nu\to \nu_1,\quad a(0,t_2)_*\nu\to \nu_2,\quad a(t_1,t_2)_*\nu\to \mu
$$
as $t_1,t_2\to\infty$.
Theorem \ref{prop:ED} below gives
a generalization of this property.
For $I\subset\{1,2\}$, we set
$$Y_{I}:=\begin{cases}
Y\quad\mbox{if } I=\emptyset, \\
Y_{i}\quad\mbox{if } I=\{i\}\mbox{ with }i=1,2, \\
X\quad\mbox{if } I=\{1,2\}.
\end{cases}$$
and denote by $\nu_I$ the corresponding measures supported on $Y_I$.
For $t\in\bR^2$, we write$
\lfloor t \rfloor_{I}:=\min\{t_{i}:i\in I\}$.

\begin{theorem}[\cite{BFG2}, Theorem 6.1]
\label{prop:ED}
Let $r\in\mb{N}$. There exist $\ell\ge 1$ and $\gamma_r>0$ such that 
for any $I_{1},\ldots,I_r\subset \{1,2\}$, $t_{1},\dotsc,t_{r}\in\bR_+^{2}$ with $(t_i)_k=0$ for $k\notin I_{i}$,
and $\varphi_{1},\dotsc,\varphi_{r}\in C_c^\infty(X)$,
\begin{multline}
\left|\int_{Y}\prod_{i=1}^{r}\varphi_{i}\circ a(t_{i})\,d\nu-\prod_{i=1}^{r}\int_{Y_{I_{i}}}\varphi_{i}\,d\nu_{I_{i}}\right| \\
\ll_r e^{-\gamma_{r} \min\big(\min_{i}\lfloor t_{i}\rfloor_{I_{i}},\, \min_{i\neq j}|t_{i}-t_{j}|\big)} \prod_{i=1}^{r}\|\varphi_{i}\|_{C^\ell}.
\end{multline}
\end{theorem}
In the rest of the paper, we work with fixed sufficiently large $\ell$ such that the above estimates apply.

\section{Moments of Siegel transforms}
\label{sec:appest}

\subsection{Basic estimates}
We start by recalling the following classical formulas 
for the Siegel transform.

\begin{theorem}[Siegel \cite{sie}]
\label{Thm_Siegel}
Let $f:\mb R^3\to \mb R$ be a bounded Riemann-integrable function with compact support.  
Then $\widehat{f} \in L^1(X)$ and
\[
\int_{X} \widehat{f} \,  d\mu = \int_{\bR^3} f(\ul{z}) \,  d\ul{z}.
\]
\end{theorem}

\begin{theorem}[Rogers \cite{rog}]
\label{thm:Rogers}
Let $f,g:\mb R^3\to\mb R$ be nonnegative Borel functions with compact support.  
Then $\widehat f\cdot\widehat g\in {L}^1(X)$ and 
\begin{align*}
\int_{X} \widehat{f}\cdot\widehat{g} \,  d\mu =&\left(\int_{\bR^3} f(\ul{z}) \,  d\ul{z}\right)\cdot \left( \int_{\bR^3} g(\ul{z}) \,  d\ul{z} \right)\\
&+\sum_{p,q\in\mb \bN} \zeta(3)^{-1}\int_{\bR^3} f(p\ul {z})\cdot g(q \ul{z}) \,  d\ul{z}
+\sum_{p,q\in\mb \bN} \zeta(3)^{-1} \int_{\bR^3} f(p\ul{z})\cdot g(-q \ul{z}) \,  d\ul{z}.
\end{align*}
\end{theorem}

In particular, using the Cauchy--Schwarz inequality, we get a following estimate:
\begin{cor}
\label{cor:Rogers}
For a nonnegative Borel function $f:\mb R^3\to\mb R$ with compact support,
$$
\int_{X}\widehat f^2\,d\mu\leq 
\left(\int_{\bR^3} f(\ul{z}) \,  d\ul{z}\right)^2+ 
2\zeta(3)^{-1}\zeta(3/2)^2\int_{\bR^3} f(\ul{z})^2 \,  d\ul{z}.$$
\end{cor}

Further, we recall that  Siegel transforms can be bounded above in terms of the successive minima of a lattice. To be precise, let $\{\ul{e}_1,\ul{e}_2,\ul{e}_3\}$ denote the standard (ordered) basis of $\bR^3$.  We extend the $\max$-norm with respect to this basis to the second exterior power $\bR^3 \wedge \bR^3$ as follows. For $\ul{u},  \ul{v} \in \bR^3$ and 
\[
\ul{w}: = \ul{u} \wedge \ul{v} = w_{12} \, \ul{e}_1 \wedge \ul{e}_2 + w_{13} \, \ul{e}_1 \wedge \ul{e}_3 + w_{23} \, \ul{e}_2 \wedge \ul{e}_3,
\]
we set $\|\ul w\|_\infty: = \max(|w_{12}|,|w_{13}|,|w_{23}|)$. 
For a (not necessarily unimodular) lattice $\Lambda$ in $\bR^3$,
the successive minima are defined as:
\begin{align*}
\label{Def_s1}
s_1(\Lambda) 
&:= 
\min\big\{ \|\ul{\lambda}\|_\infty \,  : \,  \ul{\lambda} \in \Lambda \setminus \{0\} \big\}, \\[0.2cm]
s_1^{*}(\Lambda)
&:=
\min\big\{ \|\ul{\lambda}_1 \wedge \ul{\lambda}_2\|_\infty \,  : \,  \ul{\lambda}_1 \wedge \ul{\lambda}_2 \neq 0,  \enskip \ul{\lambda}_1,\ul{\lambda}_2 \in \Lambda \setminus \{0\} \big\}, \\[0.2cm]
d(\Lambda)
&:=
\min\big\{ |\alpha| \,  : \,  
\alpha \, \ul{e}_1 \wedge \ul{e}_2 \, \wedge \ul{e}_3 = \ul{\lambda}_1 \wedge \ul{\lambda}_2 \wedge \ul{\lambda}_3 \neq 0, \enskip \ul{\lambda}_1,\ul{\lambda}_2,  \ul{\lambda}_3 \in \Lambda \setminus \{0\} \big\},
\end{align*}
Then we define the {height function} as
\begin{equation}
\label{Def_ht}
\height(\Lambda) := \min\big(s_1(\Lambda),s_1^*(\Lambda),d(\Lambda)\big)^{-1}.
\end{equation}
In terms of these notations, we have the following estimate:

\begin{lem}[W. Schmidt \cite{Schmidt68}]\label{l:siegel}
For a Borel function $f:\mb R^3\to [0,1]$ with compact support,
$$
\widehat f\ll_{\supp(f)} \height.
$$
\end{lem}

We further recall (see, for instance, \cite[Lemma 3.10]{EsMaMo})

\begin{lem}\label{l:integrable}
For $p<3$, $\height\in L^p(\mu)$  In particular, for every $L>0$ and $\delta>0$,
$$
\mu\big(\{\height \ge L \}\big)\ll_\delta L^{-3+\delta}.
$$
\end{lem}

\subsection{Estimates on homogeneous subspaces}

To carry out our argument, we will require upper bounds for the moments of Siegel transforms on the spaces $Y,Y_{1},Y_2$, and $X$. The following four propositions provide the required bounds. 
The proofs are independent of the rest of the paper and can be skipped for the first reading.

\begin{prop}
\label{prop:L1}
Let $f:\mb{R}^{3}\to [0,1]$ be a measurable function whose support is contained in $[-c,c]^{2}\times[\xi,\beta]$ for some $c>0$ and $0<\xi<\beta$.
Then for any $n\in\mb{N}^{2}_o$,
\begin{align*}
\int_{Y}\widehat{f}\circ a({n})\,d\nu 
\ll_c  \beta. 
\end{align*}
Moreover, if we assume that the function
$$
A(y):= \int_{\bR^2} f(x_1,x_2,y)\,dx_1dx_2
$$
is
piecewise $C^1$ with discontinuities in at most $s$ points 
and satisfies 
\begin{equation}\label{eq:A_cond}
|A|\le M\qand |A'|\le M'\quad\;\;\hbox{for some $M,M'>0$},
\end{equation}
then
$$
\int_{Y}\widehat{f}\circ a({n})\,d\nu=\int_{\mb R^3}f\,d\ul x+O_{s,\beta}\Big((M+M')e^{-n_1-n_2}\Big).
$$
\end{prop}

\begin{proof}
We have
\begin{align*}
\int_{Y}\widehat{f}\circ a({n})\,d\nu&=\int_{[0,1)^2}\sum_{(p_1,p_2,q)\in\mb{Z}^3}
f\big(e^{n_1}(\alpha_1q+p_1),e^{n_2}(\alpha_2q+p_2),e^{-n_1-n_2}q\big)\,d\underline{\alpha}.
\end{align*}
For fixed positive $q$, applying a change of variables, we deduce that 
\begin{align*}
&\int_{[0,1)^2}\sum_{(p_1,p_2)\in\mb{Z}^2}
f\big(e^{n_1}(\alpha_1q+p_1),e^{n_2}(\alpha_2q+p_2),e^{-n_1-n_2}q\big)\,d\underline{\alpha} \\   
=&
\int_{[0,1)^2}\sum_{h_1,h_2\in\mb Z} \sum_{k_1,k_2=0}^{q-1}  \int_{[0,1)^2}f\big(e^{n_1}((\alpha_1+h_1)q+k_1),e^{n_2}((\alpha_2+h_2)q+k_2),e^{-n_1-n_2}q\big)\,d\underline{\alpha}\\
=& \sum_{k_1,k_2=0}^{q-1}
\int_{\bR^2}   f\big(e^{n_1}(\alpha_1q+k_1),e^{n_2}(\alpha_2q+k_2),e^{-n_1-n_2}q\big)\,d\underline{\alpha}\\
=& \,e^{-n_1-n_2}\int_{\bR^2}   f\big(\alpha_1,\alpha_2,e^{-n_1-n_2}q\big)\,d\underline{\alpha}.
\end{align*}
Therefore,
\begin{align*}
\int_{Y}\widehat{f}\circ a({n})\,d\nu
&= e^{-n_1-n_2}\sum_{q\in\bZ} \int_{\bR^2}   f\big(\alpha_1,\alpha_2,e^{-n_1-n_2}q\big)\,d\ul{\alpha}
\end{align*}
When $e^{n_1+n_2}\beta<1$, we get
\begin{align*}
\int_{Y}\widehat{f}\circ a({n})\,d\nu=0,
\end{align*}
and when $e^{n_1+n_2}\beta\ge 1$, the number of terms in the sum is at most 
$\lfloor e^{n_1+n_2}\beta\rfloor$, so that 
\begin{align*}
\int_{Y}\widehat{f}\circ a({n})\,d\nu \ll_c \beta,
\end{align*}
which verifies the first bound.
The second part of the proposition follows readily from the following lemma.
\end{proof}

\begin{lem}
Let $A:\mb R\to \mb R$ be a piecewise $C^1$-function supported on the interval $[\xi,\beta]$ with discontinuities at most $s$ points 
which satisfies \eqref{eq:A_cond}.
Then for every $T\ge 1$,
$${T}^{-1}\sum_{q\in\mb Z}A\left({q}/{T}\right)=\int_{\mb R}A(y)\,dy+O\Big((s M +(\beta-\xi+2)M')T^{-1}\Big).$$
\end{lem}

\begin{proof}
We pick $q_1,q_2\in\bZ$, so that $q_1/T\le \xi <(q_1+1)/T$ and $q_2/T\le \beta <(q_2+1)/T$. Then
\begin{align*}
\int_{\mb R}A(y)\,dy=\int_{q_1/T}^{(q_2+1)/T} A(y)\,dy
=\sum_{q=q_1}^{q_2} \int_{q/T}^{(q+1)/T} A(y)\,dy.
\end{align*}
When $A$ is differentiable on $[{q/T},{(q+1)/T}]$,
it follows from the Mean Value Theorem that
$$
\int_{q/T}^{(q+1)/T} A(y)\,dy=A(q/T)/T+O(M' T^{-2}).
$$
This bound applies to all but $2s$ intervals, so that 
\begin{align*}
\int_{\mb R}A(y)\,dy&=\sum_{q=q_1}^{q_2} A(q/T)/T + O\Big(4s M/T+( T(\beta-\xi)+2)M'T^{-2}\Big)\\
&=T^{-1}\sum_{q\in\bZ} A(q/T) + O\Big((s M+(\beta-\xi+2)M')T^{-1}\Big),
\end{align*}
which proves the lemma.
\end{proof}

We also establish the following non-divergence lemma for the measures
$a(n)_*\nu$:
\begin{lem} \label{lem:BnY}
For every $L>1$ and $n\in\bN^2_o$,
$$
\nu\Big (\{\textup{ht}\circ a(n)\geq L\}\Big)\ll L^{-3}+L^{-2}e^{-\lfloor n\rfloor}.
$$
\end{lem}

\begin{proof}
We have
$$
\{\underline{\alpha}\in [0,1)^2: 
\textup{ht}(a(n)\Lambda_\alpha)\geq L\}
\subset B(n,L^{-1})\cup B^*(n,L^{-1}),
$$
where
\begin{align*}
B(n,\eta) &:=\{\underline{\alpha}\in [0,1)^2: s_1(a(n)\Lambda_\alpha)\le \eta\},\\
B^*(n,\eta)&:=\{\underline{\alpha}\in [0,1)^2: s^*_1(a(n)\Lambda_\alpha)\le \eta\}.
\end{align*}
Let $\eta\in (0,1)$.
We observe that $\ul \alpha \in B(n,\eta)$ iff there exists $(p_1,p_2,q)\in\bZ^3\backslash \{0\}$ such that 
\begin{equation}\label{eq:cc1}
|q\alpha_1 -p_1|\leq \eta e^{-n_1},\;\;
|q\alpha_2 -p_2|\leq \eta e^{-n_2},\;\;
|q|\leq \eta e^{n_1+n_2}.
\end{equation}
Since $\eta<1$, $q=0$ gives no solutions $\ul \alpha$ for \eqref{eq:cc1} and moreover,  if \eqref{eq:cc1} has a solution $\ul\alpha$ for a given $q$, there are at most $O(q^2)$ possibilities $(p_1,p_2)$. 
Therefore, the number of $(p_1,p_2,q)$ such that 
\eqref{eq:cc1} has a solution $\ul\alpha$ is $O(q^2\eta e^{n_1+n_2})$.
On the other hand, for fixed $(p_1,p_2,q)$ the measure of the set of 
solutions $\ul\alpha$ for \eqref{eq:cc1} is $O(q^{-2}\eta^2 e^{-n_1-n_2})$.
Therefore, we conclude that 
$$
\Vol(B(n,\eta))\ll \eta^3.
$$

Similarly, we also have that 
$\ul \alpha \in B^*(n,\eta)$ iff there exist linearly independent vectors $(p_1,p_2,q),(p'_1,p'_2,q')\in\bZ^3\backslash \{0\}$ such that
for
\begin{align*}
\ul v(\ul \alpha) :=& (q\alpha_1-p_1,q\alpha_2-p_2,q)\wedge (q'\alpha_1-p'_1,q'\alpha_2-p'_2,q')\\
= &\big((p_2 q'-qp_2')\alpha_1 + (p_1 q'-qp_1')\alpha_2 +(p_1p_2'-p_2p_1')\big) \ul e_1\wedge \ul e_2\\
&+(qp_2'-p_2q')\ul e_2\wedge\ul e_3+(qp_1'-p_1q')\ul e_1\wedge\ul e_3,
\end{align*}
one has 
$$
|\ul v(\ul \alpha)_{12} |\le \eta e^{-n_1-n_2},\;\;
|\ul v(\ul \alpha)_{23} |\le \eta e^{n_1},\;\;
|\ul v(\ul \alpha)_{13} |\le \eta e^{n_2}.
$$
This implies that there exists $(u_1,u_2,u_3)\in\bZ^3\backslash \{0\}$ such that 
\begin{equation}\label{eq:cc2}
|u_1\alpha_1+u_2\alpha_2+u_3 |\le \eta e^{-n_1-n_2},\;\;
|u_1 |\le \eta e^{n_1},\;\;
|u_2 |\le \eta e^{n_2}.
\end{equation}
Since $\eta<1$, when $u_1=u_2=0$, no solutions $\ul\alpha$ of \eqref{eq:cc2}
exist, and for given $(u_1,u_2)$, there are $O\big(\max(|u_1|,|u_2|)\big)$ integers $u_3$ such that solutions $\ul\alpha$ of \eqref{eq:cc2} exist.
For fixed $(u_1,u_2,u_3)$ with $(u_1,u_2)\ne 0$, the measure of the set of solutions $\ul\alpha$
satisfying \eqref{eq:cc2} is $O\big(\max(|u_1|,|u_2|)^{-1}\eta e^{-n_1-n_2} \big)$.
We consider the cases:
\begin{itemize}
\item When $\eta e^{n_1}<1$ and $\eta e^{n_2}<1$, one gets $u_1=u_2=0$,
so that \eqref{eq:cc2} has no solutions.
\item When $\eta e^{n_1}\ge 1$ and $\eta e^{n_2}<1$, $u_2=0$ and there are 
$O(\eta e^{n_1})$ possibilities for $u_1$, so that 
the measure of the set of solutions $\ul\alpha$
 is $O\big(\eta^2 e^{-n_2} \big)$.
\item When $\eta e^{n_1}< 1$ and $\eta e^{n_2}\ge 1$, 
similarly the measure of the set of solutions $\ul\alpha$
 is $O\big(\eta^2 e^{-n_1} \big)$.
\item When $\eta e^{n_1}\ge 1$ and $\eta e^{n_2}\ge 1$, 
there are 
$O(\eta^2 e^{n_1+n_2})$ possibilities for $(u_1,u_2)$, so that 
the measure of the set of solutions $\ul\alpha$
 is $O\big(\eta^3\big)$.
\end{itemize}
Ultimately, we conclude that 
$$
\Vol(B^*(n,\eta))\ll \eta^3+\eta^2 e^{-\lfloor n\rfloor},
$$
which completes the proof of the lemma.
\end{proof}

\begin{prop}
\label{prop:L2}
Let $f:\mb{R}^{3}\to [0,1]$ be a measurable function whose support is contained in the set $[-c,c]^{2}\times[\xi,\beta]$ for some $c>0$ and  $0<\xi<\beta\ll 1$.
Then, for any $n\in\mb{N}^{2}_o$,
$$
\int_{Y}\widehat{f}^{2}\circ a({n})\,d\nu\ll_c \beta\cdot\max \big(1,\ln(\beta/\xi)\big)^2 \cdot\max\left(1,|n| e^{-\lfloor {n} \rfloor}\right).
$$
\end{prop}

\begin{proof}
We have 
\begin{align*}
\int_{Y}\widehat{f}^2\circ a({n})\,d\nu=\sum_{q_{1},q_{2}\in\mb{Z}}\sum_{\ul p_{1},\ul p_{2}\in\mb{Z}^{2}} \int_{[0,1)^2} f\circ a({n})(q_1\underline{\alpha} +  \ul{p}_1 ,q_{1})f\circ a({n})(q_2\underline{\alpha}+ \ul p_2 ,q_{2})\,d\underline{\alpha}.
\end{align*}
Let $\Delta_n:=[-e^{-n_1} c,e^{-n_1} c]\times [-e^{-n_2} c,e^{-n_2} c]$.
Then when $q_1,q_2\in [\xi e^{n_1+n_2},\beta e^{n_1+n_2}]$,
\begin{align*}
&\int_{[0,1)^2} f\circ a({n})(q_1\underline{\alpha} +  \ul{p}_1 ,q_{1})f\circ a({n})(q_2\underline{\alpha}+ \ul p_2 ,q_{2})\,d\underline{\alpha}\\
\le &
\int_{[0,1)^2}\chi_{\Delta_n}(\underline{\alpha}q_{1} +  \ul p_1 )\chi_{\Delta_n}(\underline{\alpha}q_{2}+ \ul p_2 )\,d\underline{\alpha}
= \Vol\left(q_1^{-1}(\Delta_{n}- \ul p_1)\cap q_2^{-1}(\Delta_{n}- \ul p_2)\cap[0,1)^{2}\right).
\end{align*}
For the complementary range of $q_1,q_2$, this integral is zero. 
In particular, we may assume that 
$\beta e^{n_1+n_2}\geq 1.$
We note that there are at most $\lfloor \beta e^{n_1+n_2}\rfloor^2$
possibilities for $(q_1,q_2)$ that give a non-trivial contribution.

We write $q_1=dq_1'$ and $q_2=dq_2'$ with $d:=\gcd(q_1,q_2)$.
Then
$q_{2} \ul p_1 -q_{1} \ul p_2 =d\ul k$ for $\ul k\in \bZ^2$.
We group the summands according to the parameter $\ul k$.
We note that all the solutions $(\ul x,\ul y)$ of the equation
$q_{2}\ul x-q_{1}\ul y=d\ul k$ are of the form
$$
\ul x = x_{\ul k}+q_{1}'\ul l,\;\;  \ul y = y_{\ul k}+q_{2}'\ul l\quad\hbox{with  $\ul l\in \bZ^2,$}
$$
where $(x_{\ul k},y_{\ul k})$ is a fixed solution of this equation.
Therefore,
\begin{align*}
&\sum_{\ul p_{1},\ul p_{2}\in\mb{Z}^{2}} \Vol\left(q_1^{-1}(\Delta_{n}- \ul p_1)\cap q_2^{-1}(\Delta_{n}- \ul p_2)\cap[0,1)^{2}\right) \\
=&
\sum_{\ul k\in\mb{Z}^{2}}\sum_{\ul l\in\mb{Z}^{2}}\Vol\left(\left(q_1^{-1}(\Delta_{n}- x_{\ul k})-{\ul l}/d\right)\cap\left(q_2^{-1}(\Delta_{n}- y_{\ul k})-{\ul l}/{d}\right)\cap[0,1)^{2}\right)\\
=&
\sum_{\ul k\in\mb{Z}^{2}}\sum_{\ul l\in\mb{Z}^{2}}\Vol\left(q_1^{-1}(\Delta_{n}- x_{\ul k})\cap q_2^{-1}(\Delta_{n}- y_{\ul k})\cap([0,1)^{2}+\ul l/d)\right)\\
=&
\sum_{\ul k\in\mb{Z}^{2}}d^2\Vol\left(q_1^{-1}(\Delta_{n}- x_{\ul k})\cap q_2^{-1}(\Delta_{n}- y_{\ul k})\right)
=
\sum_{\ul k\in\mb{Z}^{2}}\Vol\left({q'_1}^{-1}\Delta_{n}\cap {q'_2}^{-1}(\Delta_{n}+\ul k/{q'_1})\right).
\end{align*}
We note that in the above sum only vectors
\begin{equation}\label{eq:kk}
\ul k\in q_2'\Delta_n-q_1'\Delta_n\subset 2\max(q_1',q_2')\Delta_n
\end{equation}
give non-trivial
contributions, and 
$$
\Vol\left({q'_1}^{-1}\Delta_{n}\cap {q'_2}^{-1}(\Delta_{n}+\ul k/{q'_1})\right)\ll_c \max(q'_1,q'_2)^{-2}e^{-n_1-n_2}.
$$

Finally, we have to sum the above bounds over $q_1,q_2\in [\xi e^{n_1+n_2},\beta e^{n_1+n_2}]$ to estimate the Siegel transform. Let as consider several cases:
\begin{itemize}
\item When $\max(q'_1,q'_2)\ge \max(e^{n_1},e^{n_2})$, there are
$O_c(\max(q'_1,q'_2)^{2}e^{-n_1-n_2})$ possibilities for $\ul k$ satisfying \eqref{eq:kk}, so that 
$$
\sum_{\ul k\in\mb{Z}^{2}}\Vol\left({q'_1}^{-1}\Delta_{n}\cap {q'_2}^{-1}(\Delta_{n}+\ul k/{q'_1})\right)
\ll_c (e^{-n_1-n_2})^2,  
$$
and the sum over $q_1,q_2\in [\xi e^{n_1+n_2},\beta e^{n_1+n_2}]$ gives $O_c(\beta^2)$.

\item When $\max(q'_1,q'_2)\le \min(e^{n_1},e^{n_2})$,
there are
$O_c(1)$ possibilities for $\ul k$ satisfying \eqref{eq:kk}, so that 
$$
\sum_{\ul k\in\mb{Z}^{2}}\Vol\left({q'_1}^{-1}\Delta_{n}\cap {q'_2}^{-1}(\Delta_{n}+\ul k/{q'_1})\right)
\ll_c \max(q'_1,q'_2)^{-2}e^{-n_1-n_2}, 
$$
and we obtain the sum
\begin{align*}
&\sum_{d\le \beta e^{n_{1}+n_{2}}} \left(\sum_{\xi e^{n_{1}+n_{2}}/d\le q_1',q_2'\le \beta e^{n_{1}+n_{2}}/d} \max(q'_1,q'_2)^{-2}e^{-n_1-n_2}\right)\\
\le&\, e^{-n_1-n_2} \sum_{d\le \beta e^{n_{1}+n_{2}}}\left( \sum_{\xi e^{n_{1}+n_{2}}/d\le q_1',q_2'\le \beta e^{n_{1}+n_{2}}/d} {q'_1}^{-1} {q'_2}^{-1}\right)
\ll \beta \cdot\max\left( 1,\ln({\beta}/{\xi})\right)^{2}.
\end{align*}

\item 
When $e^{n_2}\le \max(q'_1,q'_2)\le e^{n_1}$,
there are
$O_c\big(\max(q_1',q_2')e^{-n_2}\big)$ possibilities for $\ul k$ satisfying \eqref{eq:kk}, so that 
$$
\sum_{\ul k\in\mb{Z}^{2}}\Vol\left({q'_1}^{-1}\Delta_{n}\cap {q'_2}^{-1}(\Delta_{n}+\ul k/{q'_1})\right)
\ll_c \max(q'_1,q'_2)^{-1}e^{-n_1-2n_2}.
$$
In this case the greatest common divisor $d=q_i/q_i'$ satisfies $d\le \beta e^{n_1}$.
Thus, we obtain the sum
\begin{align*}
&\sum_{d\le \beta e^{n_{1}}}\left( \sum_{\xi e^{n_{1}+n_{2}}/d\le q_1',q_2'\le \beta e^{n_{1}+n_{2}}/d} \max(q'_1,q'_2)^{-1}e^{-n_1-2n_2}\right)\\
\le&\, e^{-n_1-2n_2} \sum_{d\le \beta e^{n_{1}}} \frac{\beta e^{n_1+n_2}}{d}\left(\sum_{\xi e^{n_{1}+n_{2}}/d\le q_1',q_2'\le \beta e^{n_{1}+n_{2}}/d} {q'_1}^{-1} {q'_2}^{-1}\right)\\
\ll&\, \beta\cdot\max\left( 1,\ln({\beta}/{\xi})\right)^{2}\cdot n_1e^{-n_2}.
\end{align*}

\item The remaining case is handled similarly, and we obtain the bound
$$
\ll_c \beta\cdot\max\left( 1,\ln({\beta}/{\xi})\right)^{2}\cdot n_2e^{-n_1}.
$$
\end{itemize}
Combining these cases, we get the proposition.
\end{proof}

\begin{prop}
\label{prop:alphaintoverY}
Let $f:\mb{R}^{3}\to [0,1]$ be a function supported on $[-c,c]^{2}\times[-\beta ,\beta]$ for some $c\in (0,1)$ and $0<\beta\ll 1$. 
Then, for all $p\in [1,2)$ and $n\in\mb{N}^{2}_o$, 
$$
\int_{Y_{1}}\widehat{f}^{p}\circ a({n})\,d\nu_1\ll_{c,p} \beta^{p/2}.
$$
An analogous result also holds for $Y_{2}$.
\end{prop}

\begin{proof}
We pick $\xi=\xi(c)$ such that $r:=\sqrt{c^2+\xi^2}<1$.
Let $\chi$ be the characteristic function of $[-c,c]^2\times [-\xi,\xi]$.
For $\theta>0$, we write ${\sf b}(\theta):=\hbox{diag}(1,1,\theta^{-1})$. Then
$$
\int_{Y_1} \widehat f^{p}\circ a(n)\, d\nu_1\le \int_{Y_1} \widehat \chi^{p}\circ {\sf b}(\beta/\xi) \circ a(n)\, d\nu_1.
$$
We recall that $Y_1=G_1\Gamma$. The elements $g\in G_1$ have the
Iwasawa decomposition
$$
g={\sf k}_\theta {\sf a}_s {\sf n}_{\ul x},
$$
where
\begin{align*}
{\sf k}_{\theta}&:=\begin{pmatrix}
\cos\theta & 0 & \sin\theta \\
0 & 1 & 0 \\
-\sin\theta & 0 & \cos\theta
\end{pmatrix},\quad\hbox{for $\theta\in [0,2\pi)$,}\\
{\sf a}_{s}&:=\hbox{diag}(
e^{s},1,e^{-s})
,\quad\hbox{for $s\in\mb R$},\\
{\sf n}_{ \underline{x}}&:=\begin{pmatrix}
1 & 0 & x_1 \\
x_{2} & 1 & x_{3} \\
0 & 0 & 1
\end{pmatrix},\quad\hbox{for $\underline{x}:=(x_1,x_2,x_3)\in\mb R^3$.}
\end{align*}
The Haar measure on $G_1$ is given by
$$
\int_{G_1}f\, dm_{G_1}=\int_{[0,2\pi)\times \bR\times \bR^3 } f({\sf k}_\theta {\sf a}_s {\sf n}_{\ul x})\, e^{2s}d\theta ds d\ul  x, \quad f\in C_c (G_1).
$$
For $\theta>0$, we write ${\sf b}'(\theta):=\hbox{diag}(\theta^{-1/2},1,\theta^{-1/2})$. 
Since measure $\nu_1$ is $a_s$-invariant, we may assume that $$
a(n)=\hbox{diag}(e^{-n_2/2}, e^{n_2}, e^{-n_2/2})
$$
and
replace $\sf b(\beta/\xi)$ by $\sf b'(\beta/\xi)$.
Then $a(n)$ and ${\sf b}'(\beta/\xi)$ commute with ${\sf k}_\theta$.
By Lemma~\ref{l:siegel},
$$
\int_{Y_1} \widehat \chi^{p}\circ {\sf b}' \circ a(n)\, d\nu_1
\ll_c  
\int_{Y_1} \hbox{ht}^{p}\circ {\sf b}'(\beta/\xi)\circ a(n)\cdot \omega\circ {\sf b}'(\beta/\xi)\circ a(n)\, d\nu_1,
$$
where 
$\omega$ denotes the characteristic function of the set 
$\{\Lambda:\, \Lambda\cap [-c,c]^2\times [-\xi,\xi]\ne \{0\}\}$.
We note that for every $\theta$,
$$
{\sf k}_\theta\cdot \big( [-c,c]^2\times [-\xi,\xi])\subset \{x\in\bR^3:\,\|x\|_\infty\le r\},
$$
so that  $\omega\circ {\sf k}_\theta \le \tilde \omega,$
where $\tilde \omega$ denotes the characteristic function of $\{s_1\le r\}$.
Let 
$$
\Pi:=[0, 2\pi)\times [-\infty,s_0]\times [-1/2,1/2]^3\qand
\Sigma :=\{{\sf k}_\theta {\sf a}_s{\sf n}_{\ul x}:\, (\theta, s,\ul x)\in \Pi\}.
$$
This is a Siegel set for $G_1$, and when $s_0$ is sufficiently large, $G_1\Gamma= \Sigma \Gamma.$ Therefore,
since ${\sf k}_\theta$ runs over a compact set,
\begin{align*}
&\quad\int_{Y_1} \widehat \chi^{p}\circ {\sf b}'(\beta/\xi)\circ a(n)\, d\nu_1\\
&\ll_c \int_\Pi
\hbox{ht}^{p}({\sf b}'(\beta/\xi)a(n)
{\sf k}_\theta {\sf a}_s{\sf n}_{\ul x}\Gamma)
\omega({\sf b}'(\beta/\xi)a(n){\sf k}_\theta 
 {\sf a}_s{\sf n}_{\ul x}\Gamma)
\, e^{2s}d\theta ds d\ul  x\\
&=\int_\Pi
\hbox{ht}^{p}({\sf k}_\theta {\sf b}'(\beta/\xi)a(n)
 {\sf a}_s{\sf n}_{\ul x}\Gamma)
 \omega({\sf k}_\theta {\sf b}'(\beta/\xi)a(n)
 {\sf a}_s{\sf n}_{\ul x}\Gamma)
 \, e^{2s}d\theta ds d\ul  x \\
 &\ll 
\int_{-\infty}^{s_0}\int_{[-1/2,1/2]^3}
\hbox{ht}^{p}({\sf b}'(\beta/\xi)a(n)
 {\sf a}_s{\sf n}_{\ul x}\Gamma)
 \tilde\omega({\sf b}'(\beta/\xi)a(n)
 {\sf a}_s{\sf n}_{\ul x}\Gamma)
 \, e^{2s} ds d\ul  x.
\end{align*}
We estimate this integral separately on the regions, where
$$
\hbox{ht}=d^{-1},\quad\quad \hbox{ht}=s_1^{-1},\quad\quad\hbox{ht}={s^*_1}^{-1}
$$
respectively.
Since 
$$
d({\sf b}'(\beta/\xi)a(n)
 {\sf a}_s{\sf n}_{\ul x}\Gamma)=(b/\xi)^{-1},
 $$
 the integral over the first region is $O_c(b^{p})=O_c(b)$.
For $\eta>0$, we set
\begin{align*}
B(n,s,\beta,\eta)&:=\left\{ \underline{x}\in[-1/2,1/2]^{3}:\,s_{1}\left({\sf b}'(\beta)a({n}){\sf a}_{s}{\sf n}_{ \underline{x}}\Gamma\right)\leq \eta\right\},\\
B^*(n,s,\beta,\eta)&:=\left\{ \underline{x}\in[-1/2,1/2]^{3}:\, s^*_{1}\left({\sf b}'(\beta)a({n}){\sf a}_{s}{\sf n}_{ \underline{x}}\Gamma\right)\leq \eta\right\}.
\end{align*}

In the region where $\hbox{ht}=s_1^{-1}$, we apply Lemma \ref{lem:bnsl} below with $\varrho\in (0,1/2]$. Since $r<1$, the condition $n_2>\ln \eta$ is satisfied. When $p<2$, this gives the bound
\begin{align*}
\le \int_{-\infty}^{s_0} \left(\sum_{k=1}^\infty r^{kp}
\Vol\big(B(n,s,\beta,r^{k})\big)\right) e^{2s}ds
\ll_{p,\varrho} 
\beta + \beta^{1-\varrho} e^{-\varrho n_2}\ll \beta^{p/2},
\end{align*}
when $\varrho$ is chosen sufficiently small. 

In the region where $\hbox{ht}={s^*_1}^{-1}$, we have $s^*_1\le s_1\le r$ in the above integral.
We similarly estimate the corresponding integral with the help of Lemma \ref{lem:bnsl_2} below, but
we have to consider separately the sums with 
$$
2^{-k}<\beta^{-1/2}e^{m/2+s}\qand 2^{-k}\ge \beta^{-1/2}e^{m/2+s}.
$$
In the first case, Lemma \ref{lem:bnsl_2} applies, and as above for $p<2$ we get  the bound 
\begin{align*}
\le \int_{-\infty}^{s_0} \left(\sum_{k=\lfloor\log_2 r\rfloor -1}^\infty 2^{kp}
\Vol\big(B^*(n,s,\beta,2^{-k})\big)\right) e^{2s}ds
&\ll_p  \beta^2 +\beta^{5/2} e^{- n_2/2}+
\beta^{3/2} e^{-3n_2/2}\\
&\ll \beta^{3/2}.
\end{align*}
In the second case, we use just the trivial bound on the volume.
When $\beta^{1/2}e^{-m/2-s}\le r/2$, the sum over $k$ gives zero contribution.
When $\beta^{1/2}e^{-m/2-s}>r/2$, 
we get for $p<2$,
\begin{align*}
\le \int_{-\infty}^{s_0} \left(\sum_{r/2\le  2^k\le \beta^{1/2}e^{-m/2-s}} 2^{kp}\right) e^{2s}ds
\ll_p  \beta^{p/2} e^{- pm/2}\le \beta^{p/2}.
\end{align*}
This bounds applies in both situations and proves Proposition \ref{prop:alphaintoverY}.
\end{proof}

In the following two lemmas, we establish the volume bounds for $B(n,s,\beta,\eta)$ and $B^*(n,s,\beta,\eta)$ used in the proof of the previous proposition.

\begin{lem}
\label{lem:bnsl}
Let $s\in\bR$, $\beta,\eta>0$, and $n=(-m/2,m)$ with $m\in\bN_o$, $m>\ln \eta$. Then for every $\varrho\in [0,1/2]$,
$$\Vol\big(B(n,s,\beta,\eta)\big)\ll \eta^3 \beta +
\eta^{3-2\varrho} \beta^{1-\varrho}e^{-\rho m-2(1-\varrho) s}
$$
uniformly on $\varrho$.
\end{lem}

\begin{proof}
Since for $\ul v=(v_1,v_2,v_3)\in\bR^3$, 
$$
{\sf b}'(\beta)a({n}){\sf a}_s {\sf n}_{ \underline{x}}v=\big(\beta^{-1/2}e^{-m/2}e^s(v_1+x_1v_3),e^{m}(x_2v_1+v_2+x_3v_3) ,\beta^{-1/2}e^{-m/2}e^{-s}v_3\big),
$$
the set $B(n,s,\beta,\eta)$ consists of $\ul x\in [-1/2,1/2]^3$ for which
there exists $\ul v\in \bZ^3\backslash \{0\}$ satisfying
\begin{align}\label{eq:condd1}
|v_{1}+x_{1}v_{3}|\leq \eta \beta^{1/2}e^{m/2-s},\;\;
|v_{2}+x_{2}v_{1}+x_{3}v_{3}|\leq \eta e^{-m},\;\;
|v_{3}|\leq \eta \beta^{1/2}e^{m/2+s}.
\end{align}
If \eqref{eq:condd1} has a solution $\ul x\in [-1/2,1/2]^3$, then
\begin{equation}\label{eq:conddd1}
|v_{1}|\leq |v_{3}|/2+\eta \beta^{1/2}e^{m/2-s},\;\;
|v_{2}|\leq|v_{1}|/2+|v_{3}|/2+1,\;\;
|v_3|\leq \eta \beta^{1/2}e^{m/2+s}.
\end{equation}

Let us first consider the case when $|v_3|\ge \eta \beta^{1/2} e^{m/2-s}$.
Then also $|v_1|\le 3|v_3|/2$.
The volume of the set of $\ul x\in [-1/2,1/2]^3$ satisfying \eqref{eq:condd1} is estimated as
\begin{equation*}
\ll |v_{3}|^{-1}\eta \beta^{1/2}e^{m/2-s}\cdot \max(|v_1|,|v_3|)^{-1}\eta e^{-m}\le  |v_{3}|^{-2} \eta^2 \beta^{1/2}e^{-m/2-s}.
\end{equation*}
Summing these volume bounds over $\ul v$ satisfying \eqref{eq:conddd1} and $|v_1|\le 3|v_3|/2$,
we obtain
$$
\ll \sum_{\substack{|v_3|\leq \eta \beta^{1/2}e^{m/2+s} \\[0.1mm] |v_1|\le 3|v_3|/2,\,|v_{2}|\leq 2|v_{3}|+1}}
|v_{3}|^{-2} \eta^2 \beta^{1/2}e^{-m/2-s}\ll \eta^3 \beta.
$$

Next, we consider the case when $1\le |v_3|< \eta \beta^{1/2} e^{m/2-s}$.
The volume of the set of $\ul x\in [-1/2,1/2]^3$ satisfying \eqref{eq:condd1} is estimated as
\begin{equation*}
\ll \max(|v_1|,|v_3|)^{-1}\eta e^{-m}.
\end{equation*}
Summing these volume bounds over $\ul v$ satisfying \eqref{eq:conddd1},
we obtain
$$
\ll \sum_{\substack{|v_1|< 2\eta \beta^{1/2} e^{m/2-s},\, |v_3|\leq \eta \beta^{1/2}e^{m/2+s} \\[0.1mm] |v_{2}|\leq \max(|v_1|,|v_3|)+1}}
\max(|v_1|,|v_3|)^{-1}\eta e^{-m} \ll \eta^3 \beta.
$$

Finally, we are left to consider the cases when $v_{3}=0$. Let us assume first that $v_1\neq 0$. 
 The volume of the set of $\ul x\in [-1/2,1/2]^3$ satisfying \eqref{eq:condd1} is estimated as
$O\left( |v_1|^{-1} \eta e^{-m}\right)$.
Let $\varrho\in [0,1/2]$. Summing these bounds, we obtain 
\begin{align*}
\ll \sum_{\substack{1\le |v_{1}|\leq \eta \beta^{1/2}e^{m/2-s} \\[0.1mm] |v_{2}|\leq |v_1|/2+1}}
|v_1|^{-1} \eta e^{-m}
&\ll \eta e^{-m}  \sum_{1\le |v_{1}|\leq \eta \beta^{1/2}e^{m/2-s}}
|v_1|^{1-2\varrho}\\
&\ll \eta^{3-2\varrho} \beta^{1-\varrho}e^{-\rho m-2(1-\varrho) s}
\end{align*}
uniformly on $\varrho$.

Since $m>\ln \eta$, when $v_1=v_3=0$, we also have $v_2=0$, so that all the cases have been covered.
\end{proof}

\begin{lem}
\label{lem:bnsl_2}
Let $s\in\bR$, $\beta,\eta>0$, and $n=(-m/2,m)$ with $m\in\bN_o$
satisfying 
$$
\eta<\beta^{-1/2}e^{m/2+s}.
$$
Then 
$$
\Vol\big(B^*(n,s,\beta,\eta)\big)\ll 
\eta^3 \beta^2 + \eta^4 \beta^{5/2} e^{- m/2+s}+
\eta^3 \beta^{3/2} e^{-3m/2+|s|}.
$$
\end{lem}

\begin{proof}
The set $B^*(n,s,\beta,\eta)$ consists of $\ul x\in [-1/2,1/2]^3$
such that there exist linearly independent $v,v'\in\bZ^3$
satisfying 
\begin{equation}\label{eq:small}
\left \| {\sf b}'(\beta)a({n}){\sf a}_{s}{\sf n}_{ \underline{x}}(v\wedge v')\right\|_\infty\le \eta. 
\end{equation}
Since
\begin{align*}
{\sf n}_{\ul x}(v\wedge v') =&
\Big((v_1v_2'-v_2v_1')+(x_3-x_1x_2)(v_1v_3'-v_3v_1')+x_1(v_3v_2'-v_2v_3')\Big) \ul e_1\wedge \ul e_2\\
&+\Big( (v_2v_3'-v_3v_2')+x_2(v_1v_3'-v_3v_1') \Big) \ul e_2\wedge \ul e_3
+\Big( v_1v_3'-v_3v_1' \Big) \ul e_1\wedge \ul e_3
\end{align*}
and
\begin{align*}
{\sf b}'a({n}){\sf a}_{s}(\ul e_1\wedge \ul e_2)&=\beta^{-1/2}e^{m/2+s}\,\ul e_1\wedge \ul e_2,\\
{\sf b}'a({n}){\sf a}_{s}(\ul e_2\wedge \ul e_3)&=\beta^{-1/2}e^{m/2-s}\, \ul e_2\wedge \ul e_3,\\
{\sf b}'a({n}){\sf a}_{s}(\ul e_1\wedge \ul e_3)&=\beta^{-1}e^{-m}\, \ul e_1\wedge \ul e_3.
\end{align*}
The condition \eqref{eq:small} implies existence of $\ul u\in \bZ^3\backslash \{0\}$ such that 
\begin{equation}\label{eq:small2}
|u_1+(x_3-x_1x_2)u_2+x_1u_3|\le \eta \beta^{1/2}e^{-m/2-s},\;
|u_3+x_2u_2|\le \eta \beta^{1/2}e^{-m/2+s},\; 
|u_2|\le \eta \beta e^{m}.
\end{equation}
Let us denote by $E_{\ul u}$ the set of $\ul x\in [-1/2,1/2]^3$ satisfying 
\eqref{eq:small2} and by 
$F_{\ul u}$
the set of $\ul y\in [-1,1]^3$ satisfying
\begin{equation}\label{eq:small3}
|u_1+y_3u_2+y_1u_3|\le \eta \beta^{1/2}e^{-m/2-s},\;
|u_3+y_2u_2|\le \eta \beta^{1/2}e^{-m/2+s},\; 
|u_2|\le \eta \beta e^{m}.
\end{equation}
It follows from a change of variables that 
$$
\Vol(E_{\ul u})\le \Vol(F_{\ul u}).
$$
Now we estimate $\sum_{\ul u\in \bZ^3\backslash \{0\}} \Vol(F_{\ul u})$. We note that if $F_{\ul u}\ne 0$, then
$u_1,u_2,u_3$ satisfy the conditions
\begin{equation}\label{eq:small_cond}
|u_1|\le |u_2|+|u_3|+\eta \beta^{1/2}e^{-m/2-s},\;\;
|u_3|\le |u_2|+\eta \beta^{1/2}e^{-m/2+s},\; \;
|u_2|\le \eta \beta e^{m},
\end{equation}
which we assume from now on. Further, we note that when $u_2=u_3=0$, it follows from the assumption on $\eta$
that also $u_1=0$, which is not possible,
so that $(u_2,u_3)\ne 0$.
We analyze several cases separately.

When $u_2=0$ and $u_3\ne 0$, we have 
$$
|u_3|\le \eta \beta^{1/2}e^{-m/2+s},\quad
|u_1|\le \eta \beta^{1/2}e^{-m/2}(e^s+e^{-s}), 
$$
and
\begin{align*}
\Vol(F_{\ul u})
&\ll  |u_3|^{-1} \eta \beta^{1/2}e^{-m/2-s}.
\end{align*}
Summing these volume estimates, we obtain
\begin{equation}\label{eq:ee1}
\ll \eta \beta^{1/2}e^{-m/2+s}\cdot \eta \beta^{1/2}e^{-m/2}(e^s+e^{-s})\cdot 
\eta \beta^{1/2}e^{-m/2-s}\ll \eta^3 \beta^{3/2} e^{-3m/2+|s|}. 
\end{equation}
Now it remains to consider the cases when $u_2\ne 0$.
In this situation, we obtain
\begin{align*}
\Vol(F_{\ul u})
&\ll \max(|u_2|, |u_3|)^{-1} \eta \beta^{1/2}e^{-m/2-s}\cdot |u_2|^{-1} \eta \beta^{1/2}e^{-m/2+s}\\
&= 
\max(|u_2|, |u_3|)^{-1} |u_2|^{-1} \eta^2 \beta e^{-m}.
\end{align*}
Our assumption also implies that $|u_1|\le 3\max(|u_2|, |u_3|)$, so that
summing the above volume bounds, we obtain
\begin{align}
&\ll \sum_{1\le  |u_2|\le \eta \beta e^{m}} (|u_2|+\eta \beta^{1/2}e^{-m/2+s}) |u_2|^{-1} \eta^2 \beta e^{-m}\nonumber \\
&\ll \eta^3 \beta^2 + \eta^3 \beta^{3/2} e^{-3 m/2+s} \sum_{1\le  |u_2|\le \eta \beta e^{m}} |u_2|^{-1}
\ll \eta^3 \beta^2 + \eta^4 \beta^{5/2} e^{- m/2+s}.\label{eq:ee2}
\end{align}
Combining \eqref{eq:ee1} and \eqref{eq:ee2}, we obtain the lemma.
\end{proof}

\begin{prop}
\label{prop:alphaintoverX}
Let $f:\mb{R}^{3}\to [0,1]$ be a function supported on $[-c,c]^{2}\times[-\beta,\beta]$ for some $c>0$ and $0<\beta\ll 1$. 
Then for $p<3$,
$$
\int_{X}\widehat{f}^{p}\,d\mu\ll_{c,p} \beta^{p/3}.
$$
\end{prop}

\begin{proof}
As in the previous proof,
$$
\int_{X} \widehat f^{p}\, d\mu\le \int_{X} \widehat \chi^{p}\circ {\sf b}(\beta) \, d\mu\ll_c\int_{X} \hbox{ht}^{p}\circ {\sf b}(\beta) \, d\mu.
$$
where
$\chi$ is the characteristic function of $[-c,c]^2\times [-1,1]$ and ${\sf b}(\beta):=\hbox{diag}(1,1,\beta^{-1})$.
By invariance of $\mu$,
$$
\int_X\hbox{ht}^{p}\circ {\sf b}(\beta) \, d\mu=\int_X\hbox{ht}^{p}\circ {\sf b}''(\beta) \, d\mu,
$$
where ${\sf b}''(\beta):=\hbox{diag}(\beta^{-1/3},\beta^{-1/3},\beta^{-1/3})$.
We have 
$$
s_1\circ {\sf b}''(\beta)= \beta^{-1/3}s_1,\quad
s^*_1\circ {\sf b}''(\beta)= \beta^{-2/3}s^*_1,\quad
d\circ {\sf b}''(\beta)= \beta^{-1}d,
$$
so that 
$\hbox{ht}\circ {\sf b}''(\beta)\ll \beta^{1/3}\,\hbox{ht}$.
Since $\hbox{ht}$ is $L^p$-integrable on $X$, this implies the proposition.
\end{proof}

\section{Estimating the variance: first step}
\label{sec:var1}

In this section, we proceed with analyzing the variance of averages
of functions over the sets $\cF_\Omega$, defined in \eqref{Def_FT}.
This general result will be subsequently used in Sections~\ref{sec:appccs}--\ref{sec:CompVar}. Recall that $b \in (0,1)$ denotes a parameter in the definition 
of the set $\Omega$.

\begin{prop}\label{l:var_general}
Let $\phi_n\in C_c^\infty(X)$, $n\in\cF_\Omega$, be such that uniformly on $n$, 
\begin{enumerate}
\item[(i)] $\int_Y \phi_n\circ a(n)\, d\nu=0$,
\item[(ii)]  $\int_Y \phi_n^2 \circ a(n)\, d\nu =O_\delta\left((\ln T)^\delta e^{\delta|n|}\right)$ for any $\delta>0$,
\item[(iii)] for any $p\in [1,2)$ and $i=1,2$, 
$\int_{Y_i} |\phi_n|^p \circ a(n)\, d\nu_i =O_p\left(b^{p/2}\right)$,
\item[(iv)] $\|\phi_n\|_{C^\ell}=O\big((\ln T)^{\sigma}\big)$ for some fixed $\sigma>0$.
\end{enumerate}
Then for any $B\ge B(\sigma)>0$, $B_1\ge B_1(B,\sigma)>0$ and
$B_0\ge B_0(B,B_1,\sigma)>0$,
\begin{align*}
\int_Y\left(\sum_{n\in\cF_\Omega} \phi_n\circ a(n)\right)^2\,d\nu=\sum_{(n,k)\in \cF_\Omega^\triangledown} &\int_X 
\phi_n\circ a(n)\cdot \phi_k\circ a(k)\, d\mu \\
&+
O_{c,\delta}\Big((1+(\ln T+\ln(b/a)) b^{1-\delta}) (\ln T)^{\delta}\Big)
\end{align*}
for all $\delta>0$,
where $\cF_\Omega^\triangledown=\cF_\Omega^\triangledown(B,B_0,B_1)$ denotes the subset of $(n,k)\in \cF_\Omega\times \cF_\Omega$
satisfying 
\begin{equation}\label{eq:ftri}
|n-k|< B\ln\ln T,\quad
|n|\geq B_0\ln\ln T,\quad \lfloor n \rfloor\ge  {B_1}\ln\ln T.
\end{equation}
\end{prop}

\begin{proof}
Throughout the proof, we use the bound \eqref{eq:cardQRT} which, in particular, implies that 
the number of terms in the sums below is $O_c\left((\ln T)^4\right)$.

For a parameter $B>0$ to be specified later, we write
\begin{equation}\label{eq:summ1}
\left(\sum_{n\in\cF_\Omega} \phi_n\circ a(n)\right)^2=\Phi^++\Phi^-,
\end{equation}
where $\Phi^+$ denotes the sum over the products $\phi_n\circ a(n)\cdot \phi_k\circ a(k)$
with $n,k\in \cF_\Omega$ satisfying 
\begin{equation}\label{eq:phi_plus}
|n-k|< B\ln\ln T,
\end{equation}
and $\Phi^-$ denotes
the remaining sum with 
\begin{equation}\label{eq:minus}
|n-k|\ge B\ln\ln T.
\end{equation}

Since $\int_Y \phi_{n}\circ a(n)\, d\nu=0$, it follows from Theorem \ref{prop:integrability} and (iv) that
when \eqref{eq:minus} holds
\begin{align*}
\left|\int_{Y}\phi_{n}\circ a({n})\cdot \phi_{k}\circ a(k)\,d\nu\right|&\ll \|\phi_{n}\|_{C^\ell}\|\phi_{k}\|_{C^\ell}
(\ln T)^{-\eta_2 B}
\ll  (\ln T)^{2\sigma -\eta_2 B},
\end{align*}
so that
$$
\left| \int_Y \Phi^-\, d\nu \right|\ll |\cF_\Omega|^2 (\ln T)^{2\sigma -\eta_2 B}\ll_c (\ln T)^{4+2\sigma -\eta_2 B}.
$$
This contribution is negligible when the parameter $B$ is chosen sufficiently large depending on $\sigma$.

It remains to deal with $\Phi^+$, which we write as
\begin{equation}\label{eq:summ2}
\Phi^+=\Phi^{(1)}+\Phi^{(2)}+\Phi^{(3)},
\end{equation}
where $\Phi^{(1)}$
is the subsum of $\Phi^+$ given by the condition 
\begin{equation}\label{eq:phi_1}
|n|< B_0\ln\ln T,
\end{equation}
$\Phi^{(2)}$ 
is the subsum of $\Phi^+$ given by the condition 
\begin{equation}\label{eq:phi_2}
|n|\geq B_0\ln\ln T\qand \lfloor n \rfloor< {B_1}\ln\ln T,
\end{equation}
and 
$\Phi^{(3)}$
is the subsum given by the condition 
\begin{equation}\label{eq:phi_3}
|n|\geq B_0\ln\ln T\qand \lfloor n \rfloor\ge  {B_1}\ln\ln T,
\end{equation}
where $B_0$ and $B_1$ are constants yet to be chosen. 

To estimate $\Phi^{(1)}$, we observe that 
$$
\left|\int_{Y}\phi_{n}\circ a(n)\cdot \phi_{k}\circ a(k)\,d\nu\right|\le 
\big\|\phi_{n}\circ a(n)\big\|_{L^2(\nu)}\cdot \big\|\phi_{k}\circ a(k)\big\|_{{L}^2(\nu)}.
$$
Since the parameter domain for $\Phi^{(1)}$ is given by \eqref{eq:phi_plus} and \eqref{eq:phi_1}, we have 
$$
|n|,|k|\leq (B+B_0)\ln\ln T.
$$
Therefore, the number of summands in $\Phi^{(1)}$ is $O_{B,B_0}\left((\ln\ln T)^4\right)$, and
it follows from (ii) that 
$$
\left\|\phi_{n}\circ a(n)\right\|_{{L}^2(\nu)},\;
\left\|\phi_{k}\circ a(k)\right\|_{{L}^2(\nu)}
\ll_{\delta} (\ln T)^{(1+B+B_0)\delta/2}\quad\hbox{for any $\delta>0$},
$$
uniformly on $n,k$. Therefore,
$$
\left| \int_Y \Phi^{(1)}\, d\nu \right|\ll_{B,B_0,\delta'} (\ln T)^{\delta'}\quad\hbox{for any $\delta'>0$}.
$$

Next we analyze the sum $\Phi^{(2)}$. Without loss of generality we only consider the part of the sum with $n_{1}\ge n_{2}$,
so that $\lfloor n \rfloor=n_{2}$.
The other part of the sum can be analyzed similarly.
By Theorem \ref{prop:ED} with $r=1$,
\begin{align}\label{eq:stepp}
\int_{Y}\phi_{n}\circ a(n)\cdot \phi_{k}\circ a(k)\,d\nu
=&
\int_{Y}\Big(\phi_{n}\circ a(0,{n}_{2})\cdot \phi_{k}\circ a\big(k-(n_{1},0)\big)\Big) \circ a(n_{1},0)\,d\nu \\
=&
\int_{Y_1}\phi_{n}\circ a(0,{n}_{2})\cdot \phi_{k}\circ a(k-(n_{1},0))\,d\nu_1 \nonumber \\
&\quad\quad+ O\Big(e^{-\gamma_1 {n_{1}}}\big\|\phi_{n}\circ a(0,{n}_{2})\cdot \phi_{k}\circ a(k-(n_{1},0))\big\|_{C^\ell} \Big). \nonumber
\end{align}
Since the measure $\nu_1$ is $a(*,0)$-invariant,
\begin{equation}\label{eq:int00}
\int_{Y_1}\phi_{n}\circ a(0,{n}_{2})\cdot \phi_{k}\circ a(k-(n_{1},0))\,d\nu_1
=\int_{Y_1}\phi_{n}\circ a({n})\cdot \phi_{k}\circ a(k)\,d\nu_1.
\end{equation}
It follows from conditions \eqref{eq:phi_plus} and \eqref{eq:phi_2} that 
$$
n_{2}\leq B_1\ln\ln T\qand |k-(n_{1},0)|\le (B+B_1)\ln\ln T,
$$
so that using (iv), we deduce that 
\begin{align*}
\big\|\phi_{n}\circ a(0,{n}_{2})\cdot \phi_{k}\circ a(k-(n_{1},0))\big\|_{C^\ell}
&\ll\big\|\phi_{n}\circ a(0,{n}_{2})\big\|_{C^\ell} \cdot \big\|\phi_{k}\circ a(k-(n_{1},0))\big\|_{C^\ell}\\
&\ll (\ln T)^{\sigma'}
\end{align*}
for some $\sigma'>0$ depending on $\sigma,B,B_1$.
On the other hand, also $n_{1}\geq B_0\ln\ln T$.
Therefore, taking $B_0$ sufficiently large depending on
$\sigma'$ and $\gamma_1$, we obtain that the sum of the second terms in \eqref{eq:stepp}
gives a negligible contribution, and it remains to investigate the sum of integrals
\eqref{eq:int00}.
From (iii), for $p<2$,
\begin{align*}
\left|\int_{Y_1}\phi_{n}\circ a({n})\cdot \phi_{k}\circ a(k)\, d\nu_1\right|&\leq \big\|\phi_{n}\circ a({n})\big\|_{{L}^2(\nu_1)}\cdot \big\|\phi_{k}\circ a(k)\big\|_{{L}^2(\nu_1)}\\
&\le  (\ln T)^{\sigma(2-p)} \left(\int_{Y_1} |\phi_{n}|^p\circ a({n})\, d\nu_1\right)^{1/2}
\left(\int_{Y_1} |\phi_{k}|^p\circ a({k})\, d\nu_1\right)^{1/2}\\
&\ll_p (\ln T)^{\sigma(2-p)} b^{p/2}.
\end{align*}
Recalling the definition of $\cF_\Omega$ given in \eqref{Def_FT},
we see that the number of $(n,k)\in\cF_\Omega\times \cF_\Omega$ satisfying 
\eqref{eq:phi_plus} and \eqref{eq:phi_2} is 
$O_{B,B_1}\left(\big(\ln T+\ln(b/a)\big)(\ln\ln T)^3\right)$. 
Therefore, 
$$
\left|\int_Y \Phi^{(2)}\, d\nu\right|\ll_{B,B_1,\delta'} (\ln T+\ln(b/a))(\ln T)^{\delta'} b^{1-\delta'}\quad 
\hbox{for any $\delta'>0$.}
$$

Finally, we proceed with estimating $\Phi^{(3)}$, so that $(n,k)$ in the sum satisfy \eqref{eq:phi_plus} and \eqref{eq:phi_3}.
In this case, by Theorem \ref{prop:ED} with $r=1$,
\begin{align}\label{eq:phi33}
\int_{Y} \phi_{n}\circ a({n}) \cdot \phi_{k}\circ a(k)\,d\nu
=&
\int_{Y}\left(\phi_{n}\cdot \phi_{k}\circ a(k-{n})\right)\circ a({n})\,d\nu\\
=&
\int_{X}\phi_{ n} \cdot \phi_{ k}\circ a(k-{n})\,d\mu \nonumber\\
 &\quad\quad+ O\left( 
e^{-\gamma_1 \lfloor n\rfloor}
\left\|\phi_{ n} \cdot \phi_{ k}\circ a(k-{n})\right\|_{C^\ell}
\right). \nonumber 
\end{align}
By invariance of the measure $\mu$,
$$
\int_{X}\phi_{n} \cdot \phi_{k}\circ a(k-{n})\,d\mu=
\int_{X}\phi_{n}\circ a({n}) \cdot \phi_{k}\circ a(k)\,d\mu.
$$
As in the previous case in view of \eqref{eq:phi_plus},
\begin{align*}
\big\|\phi_{n}\cdot \phi_{k}\circ a(k-n)\big\|_{C^\ell}
&\ll\big\|\phi_{n}\big\|_{C^\ell} \cdot \big\|\phi_{k}\circ a(k-n)\big\|_{C^\ell}
\ll (\ln T)^{\sigma''}
\end{align*}
for some $\sigma''>0$ depending  on $B$ and $\sigma$.
Since $\lfloor n\rfloor\ge B_1\ln\ln T$, the sum of the second terms in \eqref{eq:phi33} gives a negligible contribution when $B_1$ is chosen sufficiently large depending on $\sigma''$ and $\gamma_1$. 
Finally, the sum of the first terms in \eqref{eq:phi33} gives
$$
\sum_{(n,k)\in \cF_\Omega^\triangledown} \int_X 
\phi_n\circ a(n)\cdot \phi_k\circ a(k)\, d\mu.
$$
Combining the above estimates
for $\Phi^{(1)},\Phi^{(2)},\Phi^{(3)}$, the main result follows.
\end{proof}

\section{Approximation by smooth compactly-supported functions}
\label{sec:appccs}

In view of the tesselation \eqref{eq:tess}, we work with the characteristic functions
$\chi_{\Delta_{\Omega,{n}}}$ and their Siegel transforms
$\widehat\chi_{\Delta_{\Omega,{n}}}$. We also introduce their smoothened versions (cf. \cite[Section 8]{BFG2}).
Let $\rho_\eps\in C_c^\infty(G)$ with $\eps\in (0,1)$ be a family of bump-functions satisfying
\begin{equation} \label{rhoeps}
\rho_\eps\ge 0, \quad \hbox{supp}(\rho_\eps)\subset V_\eps,\quad 
\int_G \rho_\eps\, dm=1, \quad \|\rho_\eps\|_{C^\ell} \ll_\ell \eps^{-\sigma_\ell}
\end{equation}
for some $\sigma_\ell>0$.
Taking a decreasing function $\eps_T\to 0$ as
$T \ra \infty$ (to be specified later),  we define
\[
f_{\Omega,{n}}(v) := \rho_{\eps_T} * \chi_{\Delta_{\Omega,{n}}}(v)=
\int_G \rho_{\eps_T}(g)\chi_{\Delta_{\Omega,{n}}}(g^{-1}v)\, dm(g),\quad v\in\bR^3.
\]
We note that $0\le f_{\Omega,{n}}\le 1$, and it follows from \eqref{eq:subb}
that
\begin{equation}\label{eq:bbound}
\supp(f_{\Omega,{n}})\subset  
 \left[- (1+\varepsilon_T)\ct,(1+\varepsilon_T)\ct\right]^2
 \times 
\left((1-\varepsilon_T){\at}{\ct^{-2}},(1+\varepsilon_T)e^2 {\bt }{\ct^{-2}}\right].
\end{equation}
Also it follows from Lemma \ref{l:siegel} that 
\begin{align}\label{eq:f_uniform}
\widehat f_{\Omega,{n}} \ll_c \hbox{ht}.
\end{align}
However, the Siegel transform $\widehat f_{\Omega,{n}}$ is not uniformly bounded. We will need more precise control on its $L^p$-norms.

\begin{lem}\label{l:f_l1_l2} For every $n\in \cF_\Omega$,
$$
\left\|\widehat f_{\Omega, n}\right\|_{L^1(\mu)}\ll_c b
\qand
\left\|\widehat f_{\Omega, n}\right\|_{L^2(\mu)}\ll_c b^{1/2}.
$$    
\end{lem}

\begin{proof}
We have
$$
\left\|\widehat f_{\Omega, n}\right\|_{L^1(\mu)}=\int_X \widehat \chi_{\Delta_{\Omega,n}}\, d\mu,
$$
and
$$
\left\|\widehat f_{\Omega, n}\right\|_{L^2(\mu)}
\le \left\|\rho_{\eps_T}\right\|_{L^1(m)}\cdot \left\|\widehat \chi_{\Delta_{\Omega,n}}\right\|_{L^2(\mu)}=
\left(\int_X \widehat \chi_{\Delta_{\Omega,n}}^2\, d\mu\right)^{1/2}.
$$
Therefore, the above estimates follow from Theorem \ref{Thm_Siegel} and Corollary \ref{cor:Rogers} respectively, taking \eqref{eq:subb} into account.
\end{proof}

The following result is a simplified version of \cite[Lemma 8.5]{BFG}.

\begin{lem}
\label{prop:appestL1}
When $\eps_T=O_c(a)$, for every $n\in \cF_\Omega$,
\begin{equation}
\left\|\left(\widehat{\chi}_{\Delta_{\Omega, n}}-\widehat{f}_{\Omega, n}\right)\circ a( n)\right\|_{L^1(\nu)} \\
\ll_{c} b\nonumber
\end{equation}
and moreover, 
\begin{equation}
\left\|\left(\widehat{\chi}_{\Delta_{\Omega, n}}-\widehat{f}_{\Omega, n}\right)\circ a( n)\right\|_{L^1(\nu)} \\
\ll_{c} \max\left(\varepsilon_T,-\frac{\varepsilon_T}{a}\ln\left(\frac{\varepsilon_T}{a}\right)\right)+e^{-n_1-n_2}.\nonumber
\end{equation}
\end{lem}

\begin{proof}
It follows from the properties of $\rho_\eps$ stated in \eqref{rhoeps} that 
\begin{align*}
&\left\|\left(\widehat{\chi}_{\Delta_{\Omega,n}}-\widehat{f}_{\Omega, n}\right)\circ a( n)\right\|_{L^1(\nu)}\\[1.5mm]
=&\int_{Y}\left|\widehat{\chi}_{\Delta_{\Omega, n}}(a(n)x)-\int_{G}\rho_{\varepsilon_T}(g)\widehat{\chi}_{\Delta_{\Omega, n}}(g^{-1}a(n)x)\,dm(g)\right|\,d\nu(x)\nonumber \\[1.5mm]
= & \int_{Y}\left|\int_{G}\rho_{\varepsilon_T}(g)\left(\widehat{\chi}_{\Delta_{\Omega, n}}(a(n)x)-\widehat{\chi}_{\Delta_{\Omega, n}}(g^{-1}a(n)x)\right)\,dm(g)\right|d\nu(x)\nonumber \\[1.5mm]
\leq & \int_{Y}\int_{G}\rho_{\varepsilon_T}(g)\left|\widehat{\chi}_{\Delta_{\Omega, n}}(a(n)x)-\widehat{\chi}_{\Delta_{\Omega, n}}(g^{-1}a(n)x)\right|\,dm(g)d\nu(x)\nonumber \\[1.5mm]
\leq & \sup_{g\in V_{\eps_T}}\left\|\left(\widehat{\chi}_{\Delta_{\Omega, n}}-\widehat{\chi}_{g\Delta_{\Omega, n}}\right)\circ a(n)\right\|_{L^{1}(\nu)}\leq \left\|\widehat\chi_{\partial\Delta_{\Omega, n}}\circ a(n)\right\|_{L^1(\nu)},
\end{align*}
where 
$$
\partial \Delta_{\Omega, n}:=\left(\cup_{g\in V_{\eps_T}} g\Delta_{\Omega,n}\right)\backslash \left(\cap_{g\in V_{\eps_T}} g\Delta_{\Omega,n}\right).
$$
It follows from \eqref{eq:subb} that 
\begin{equation}
\label{def:DeltaOmega}
\partial\Delta_{\Omega,{n}} 
\subset
 \left[- (1+\eps_T)\ct,(1+\eps_T)\ct\right]^2
 \times 
 \left[(1-\eps_T){\at}{\ct^{-2}},(1+\eps_T)e^2 {\bt }{\ct^{-2}}\right].
\end{equation}
Then, the first assertion follows from the first part of Proposition \ref{prop:L1}. 
We further note that the set $\partial\Delta_{\Omega,{n}}$ is contained in the union of the sets $E_s$ given in \eqref{eq:controled}. 
Since these sets are defined by finitely many linear inequalities,
the second part of Proposition \ref{prop:L1}
applies to the functions $\chi_{E_s}$.
Therefore, the second estimate follows from 
Proposition~\ref{prop:L1} and \eqref{eq:es}.
\end{proof}

Similar argument also gives

\begin{lem}
\label{prop:appestL1_2}
When $\eps_T=O_c(a)$, for every $n\in \cF_\Omega$,
\begin{equation}
\left\|\widehat{\chi}_{\Delta_{\Omega, n}}-\widehat{f}_{\Omega, n}\right\|_{L^1(\mu)} \\
\ll_{c} \max\left(\varepsilon_T,-\frac{\varepsilon_T}{a}\ln\left(\frac{\varepsilon_T}{a}\right)\right).\nonumber
\end{equation}
\end{lem}

\begin{proof}
Indeed, as in Lemma \ref{prop:appestL1}, with $\partial\Delta_{\Omega,{n}}$ defined in \eqref{def:DeltaOmega}, 
\begin{align*}
\left\|\widehat{\chi}_{\Delta_{\Omega,n}}-\widehat{f}_{\Omega, n}\right\|_{L^1(\mu)}\leq \left\|\widehat\chi_{\partial\Delta_{\Omega, n}}\right\|_{L^1(\mu)},
\end{align*}
and
$$
\left\|\widehat\chi_{\partial\Delta_{\Omega, n}}\right\|_{L^1(\mu)}\le \sum_s \left\|\widehat\chi_{E_s}\right\|_{L^1(\mu)}\ll_c
\max\left(\varepsilon_T,-\frac{\varepsilon_T}{a}\ln\left(\frac{\varepsilon_T}{a}\right)\right),
$$
where the last estimate follows from Theorem \ref{Thm_Siegel} and \eqref{eq:es}.
\end{proof}

\begin{lem}\label{p:app_l2}
When $\eps_T=O_c(a)$, for every $n\in \cF_\Omega$,
$$
\left\|\widehat f_{\Omega, {n}}-\widehat\chi_{\Delta_{\Omega,{n}}}\right\|_{{L}^2(\mu)}\ll_c
\max\left(\varepsilon_T,-\frac{\varepsilon_T}{a}\ln\left(\frac{\varepsilon_T}{a}\right)\right)^{1/2}.
$$
\end{lem}

\begin{proof}
It follows from the Jensen inequality that
\begin{align*}
&\left\|\widehat f_{\Omega, {n}}-\widehat\chi_{\Delta_{\Omega,{n}}}\right\|^2_{{L}^2(\mu)}\\
  =&\int_{X}\left(\int_{G}\rho_{\varepsilon_T}(g)\cdot \big(\widehat\chi_{\Delta_{\Omega,{n}}}(g^{-1}x)-\widehat\chi_{\Delta_{\Omega, n}}(x)\big)\,dm(g)\right)^2\,d\mu(x)\\
\le&   
\int_{X}\int_{G}\rho_{\varepsilon_T}(g)\cdot \big(\widehat\chi_{\Delta_{\Omega,{n}}}(g^{-1}x)-\widehat\chi_{\Delta_{\Omega, n}}(x)\big)^2\,dm(g)\,d\mu(x)\\
\leq & \sup_{g\in V_{\eps_T}}\left\|\widehat{\chi}_{\Delta_{\Omega, n}}-\widehat{\chi}_{g\Delta_{\Omega, n}}\right\|^2_{{L}^{2}(\mu)}\leq \left\|\widehat\chi_{\partial\Delta_{\Omega, n}}\right\|^2_{{L}^2(\mu)},
\end{align*}
where $\partial\Delta_{\Omega,{n}}$ is defined in \eqref{def:DeltaOmega}. Then as in the previous argument,
$$
\left\|\widehat\chi_{\partial\Delta_{\Omega, n}}\right\|_{{L}^2(\mu)}\le \sum_s \big\|\widehat\chi_{E_s}\big\|_{{L}^2(\mu)}\ll_c
\max\left(\varepsilon_T,-\frac{\varepsilon_T}{a}\ln\left(\frac{\varepsilon_T}{a}\right)\right)^{1/2},
$$
where we used Corollary \ref{cor:Rogers} and \eqref{eq:es} in the last estimate.
\end{proof}

We use a family of smooth functions $\eta_L : X \ra [0,1]$ with $L>1$ (cf. \cite[Lemma 4.11]{BG19}) such that
\begin{equation}
\label{etaL}
\{ \height \leq L/2 \} \subset \{ \eta_L = 1 \} \subset \supp(\eta_L) \subset 
\{ \height \leq 2L \},
\end{equation}
where $\height$ is the height function on $X$ defined in \eqref{Def_ht}, and
\begin{equation}
\label{DmetaL}
\| \eta_L \|_{C^\ell} \ll_\ell 1,  \quad \textrm{for all $L \geq 1$}.
\end{equation}
We use a parameter $L_T\to \infty$ as $T\to\infty$, to be specified later,
and introduce the following functions that will be used to 
construct approximations for the discrepancy:
\begin{align}
\psi_{\Omega,{n}} &:= \widehat\chi_{\Delta_{\Omega,{n}}} - \int_{X} \widehat\chi_{\Delta_{\Omega,n}}\, d\mu=
\widehat\chi_{\Delta_{\Omega,{n}}} - \Vol(\Delta_{\Omega,{n}})
, \label{eq:phiTRn} \\
\label{def_varphiTn}
\varphi_{\Omega,n} &:= \widehat{f}_{\Omega,n} \cdot \eta_{L_T} - \int_{Y} \left(\widehat{f}_{\Omega,n} \cdot \eta_{L_T}\right)\circ a(n) \,  d\nu.
\end{align}
It follows from \eqref{eq:f_uniform} and \eqref{etaL} that 
\begin{equation}\label{eq:phi_uniform}
|\varphi_{\Omega,n}|\ll_c L_T.
\end{equation}
We note that our definition of the function $\varphi_{\Omega,n}$ here
differs from the definition used in \cite[Sec. 8.2]{BFG} by
the constant
$$
\int_{X}\widehat{f}_{\Omega,{n}} \cdot \eta_{L_T}\,d\mu-\int_{Y}\left(\widehat{f}_{\Omega,{n}}\cdot \eta_{L_T}\right)\circ a({n})\,d\nu.
$$
Since the supports of the functions $f_{\Omega,{n}}$
are uniformly bounded (cf. \eqref{eq:bbound}), it follows from Theorem~\ref{Thm_Siegel}
that the first integral is uniformly bounded. 
By Proposition \ref{prop:L1}, the second integral is also 
uniformly bounded. Therefore, the estimate of 
\cite[Lemma 8.4]{BFG} still applies:
\begin{equation}\label{eq:norm_bound}
\big\|\widehat{f}_{\Omega,{n}} \cdot \eta_{L_T}\big\|_{C^\ell},\,
\|\varphi_{\Omega,n}\|_{C^\ell} \ll_\ell \eps_T^{-\sigma_\ell} \cdot L_T
\end{equation}
uniformly on $n$ and sufficiently large $T$ with some $\sigma_\ell>0$.

We consider the corresponding sums
\begin{align} \label{eq:sums}
\Psi_{\Omega}:=\sum_{n\in \cF_\Omega}\psi_{\Omega,{n}}\circ a({n})\qand
\Phi_\Omega := \sum_{n\in \cF_\Omega}\varphi_{\Omega,{n}}\circ a({n}).
\end{align}
It follows from \eqref{eq:tess} that for  $\ul x\in\bR^2$,
\begin{align*}
N_{a,b,T}(\ul x)
=|\Omega\cap \Lambda_{\ul x}|
=\sum_{n\in \cF_\Omega} \widehat\chi_{\Delta_{\Omega,{n}}}(a(n)\Lambda_{\ul x}),
\end{align*}
and
$$
V_\Omega:=\Vol(\Omega)=\sum_{n\in\cF_\Omega} \int_X \widehat\chi_{\Delta_{\Omega,{n}}}\, d\mu,
$$
so that 
\begin{equation}\label{eq:connn}
N_{a,b,T}(\ul x)-\Vol(\Omega)=\Psi_\Omega(\Lambda_{\ul x}).
\end{equation}
We recall from \eqref{eq;voll} that
\[
V_\Omega \sim 2 (\ln T)^2 (b-a) \quad \textrm{as $T \ra \infty$}.
\]
The following proposition shows that the sums 
$\Psi_\Omega$ and $\Phi_\Omega$ have the same asymptotic behavior on the scale of order $V_\Omega^{-1/2}$:
\begin{prop}
\label{lem:convinL1}
Suppose that some $K_0,K_1,K_2,\theta>0$,
\begin{align}\label{eq:cond_approx}
&a\leq b/2,\quad
(\ln T)^{-K_0}\leq \eps_T \le a (\ln T)^{-K_1},\\
&b^{-1/3}(\ln T)^{\theta}\le L_T \leq (\ln T)^{K_2},\;\; b\ge (\ln T)^{-2\min(1,K_1-1)+\theta}.\nonumber
\end{align}
Then
$$
\big\|\Psi_{\Omega}-\Phi_\Omega\big\|_{{L}^1(\nu)}=o\left(V_\Omega^{1/2}\right)\quad\mbox{as }T\to \infty.  
$$
\end{prop}

\begin{proof}
Let 
$$
\widehat{f}_{\Omega,{n}}^{(L_T)}:=\widehat{f}_{\Omega,{n}}\cdot(1-\eta_{L_T})\qand
g_{\Omega,{n}}^{(L_T)}:=\widehat{f}_{\Omega,{n}}^{(L_T)}- 
\int_Y \widehat{f}_{\Omega,{n}}^{(L_T)}\circ a(n)\, d\nu.
$$
We start by observing (cf. \eqref{eq:phiTRn} and \eqref{def_varphiTn}) that
\begin{align*}
(\psi_{\Omega,{n}}- \varphi_{\Omega, {n}})\circ a({n})= & \left(\widehat{\chi}_{\Delta_{\Omega, {n}}}-\widehat{f}_{\Omega,{n}}\right)\circ a( n)+\widehat{f}_{\Omega,{n}}^{(L_T)}\circ a({n}) \\
& -\int_{X}\widehat{\chi}_{\Delta_{\Omega, {n}}}\,d\mu+\int_{Y}\widehat{\chi}_{\Delta_{\Omega, {n}}}\circ a({n})\,d\nu \\
& -\int_{Y}\widehat{\chi}_{\Delta_{\Omega, {n}}}\circ a({n})\,d\nu+\int_{Y}\widehat{f}_{\Omega,{n}}\circ a({n})\,d\nu \\
& -\int_{Y}\widehat{f}_{\Omega,{n}}^{(L_T)}\circ a({n})\,d\nu,
\end{align*}
so that
\begin{align}
\left\|\Psi_{\Omega}- \Phi_{\Omega}\right\|_{L^1(\nu)}=&
\left\|\sum_{n\in\cF_\Omega}(\psi_{\Omega,{n}}- \varphi_{\Omega, {n}})\circ a({n})\right\|_{L^1(\nu)}\nonumber\\
\leq &\, 2\sum_{n\in\cF_\Omega}\left\|\left(\widehat{\chi}_{\Delta_{\Omega, {n}}}-\widehat{f}_{\Omega,{n}}\right)\circ a( n)\right\|_{L^1(\nu)}\label{eq:tanya1} \\
&\quad +\sum_{n\in\cF_\Omega}\left|\int_{X}\widehat{\chi}_{\Delta_{\Omega, {n}}}\,d\mu-\int_{Y}\widehat{\chi}_{\Delta_{\Omega, {n}}}\circ a({n})\,d\nu\right|\label{eq:tanya2} \\
&\quad +\left\|\sum_{n\in\cF_\Omega}g_{\Omega,{n}}^{(L_T)}\circ a(n)\right\|_{L^1(\nu)}.\label{eq:tanya3}    
\end{align}
We estimate each of these terms separately.
We recall (cf. \eqref{eq:cardQRT} and \eqref{eq;voll}) that 
$$
|\cF_\Omega|\ll_c (\ln T)^2\qand V_\Omega\gg (\ln T)^2 b,
$$
since $a\le b/2$ and $b\ge (\ln T)^{-2+\theta}$.

To estimate \eqref{eq:tanya1}, 
we observe that by Lemma \ref{prop:appestL1},
$$
\left\|
\left(
\widehat{\chi}_{\Delta_{\Omega,{n}}}-\widehat{f}_{\Omega,{n}}
\right)
\circ a({n})
\right\|_{{L}^1(\nu)}
\ll_c b.
$$
We use this bound only when $|n|\le B\ln\ln T$
for a parameter $B>0$. This gives $O_B\big((\ln\ln T)^2\big)$ summands.
On the other hand, when $|n|>B\ln\ln T$,
we use the second part of Lemma \ref{prop:appestL1} to deduce that
\begin{align*}
\left\|\left(\widehat{\chi}_{\Delta_{\Omega,{n}}}-\widehat{f}_{\Omega,{n}}\right)\circ a({n})\right\|_{L^1(\nu)}
&\ll_c 
\max\left(\varepsilon_T,-\frac{\varepsilon_T}{a}\ln\left(\frac{\varepsilon_T}{a}\right)\right)+ e^{-n_1-n_2}\\
&\ll_{K_1} (\ln T)^{-K_1}\ln\ln T+ (\ln T)^{-B},
\end{align*}
where used that $\eps_T \le a (\ln T)^{-K_1}$.
Therefore, we conclude that
\begin{align*}
\sum_{n\in\cF_\Omega} \left\|\left(\widehat{\chi}_{\Delta_{\Omega,{n}}}-\widehat{f}_{\Omega,{n}}\right)\circ a({n})\right\|_{L^1(\nu)}   
\ll_{c,K_1,B} (\ln\ln T)^2 b + (\ln T)^{2-K_1}\ln\ln T+ (\ln T)^{2-B},
\end{align*}
which is $o(V_\Omega^{1/2})$, when $B$ is sufficiently large 
due to our assumption $b\ge (\ln T)^{-2(K_1-1)+\theta}.$

To estimate \eqref{eq:tanya2}, we note that by Theorem \ref{Thm_Siegel},
$$
\int_{X}\widehat{\chi}_{\Delta_{\Omega, {n}}}\,d\mu=\Vol(\Delta_{\Omega, {n}})\ll_c b.
$$
We apply Proposition \ref{prop:L1} to ${\chi}_{\Delta_{\Omega, {n}}}$. In order to do so, we need to show that 
uniformly in $n$
the function 
$$
A_n(y):=\int_{\bR^2} \chi_{\Delta_{\Omega, n}}(x_1,x_2,y)\, dx_1dx_2
$$
is bounded and $C^1$ except at finitely many points, and  has bounded derivative. From \eqref{Def_DeltaTn}, for $y\in [e^{-(n_1+n_2)},Te^{-(n_1+n_2)})$,
$$A_n(y)=4 \int_{ce^{-1}}^{c}
\left|[ce^{-1},c]\cap \left[a x_1^{-1}y^{-1}, b x_1^{-1}y^{-1}\right]\right|\,dx_1,
$$
and $A_n(y)=0$ otherwise. This function is clearly piecewise $C^1$,  and 
$$
|A_n'|\ll_c 1\qand |A_n'(y)|\ll_c b\cdot y^{-2}\quad\hbox{uniformly on $n$.}
$$
We note that also $A_n(y)=0$ when $y\le c^{-2}a$, so that 
since $a\ge (\ln T)^{2(K_1-K_0)}$,
we conclude that
$$
|A_n'(y)|\ll_c b\cdot a^{-2}\le (\ln T)^{2(K_1-K_0)}\quad\hbox{uniformly on $n$ and $y$. }
$$
Therefore, it follows from Proposition \ref{prop:L1} that
\begin{equation*}
\left|\int_{Y}\widehat{\chi}_{\Delta_{\Omega,{n}}}\circ a({n})\,d\nu -\Vol(\Delta_{\Omega,{n}})\right|\ll_c (\ln T)^{2(K_1-K_0)} e^{-n_1-n_2}.    
\end{equation*}
In any case, by the first part of Proposition \ref{prop:L1}, we also have 
\begin{equation*} 
\left|\int_{Y}\widehat{\chi}_{\Delta_{\Omega,{n}}}\circ a({n})\,d\nu -\Vol(\Delta_{\Omega,{n}})\right|\ll_c b. 
\end{equation*}
Considering the cases when $|n|< B\ln\ln T$ and when $|n|\ge B\ln\ln T$, we obtain that 
$$
\sum_{n\in \cF_\Omega}\left|\int_{Y}\widehat{\chi}_{\Delta_{\Omega,{n}}}\circ a({n})\,d\nu-\Vol(\Delta_{\Omega,n})\right|\ll_B(\ln\ln T)^2 b + (\ln T)^{2+2(K_1-K_0)-B}
$$
for any $B>0$. When $B$ is sufficiently large, this gives $o(V_\Omega^{1/2})$ because of our assumption on $b$.

Finally, we deal with \eqref{eq:tanya3}. 
Let us introduce a new parameter $M_T:=(\ln T)^{K_2'}\geq L_T$ for some $K_2'>K_2$
to be chosen later and set 
$$
\widehat{f}_{\Omega,{n}}^{(L_T,M_T)}:=\widehat{f}_{\Omega,{n}}\cdot(\eta_{M_T}-\eta_{L_T})\qand
g_{\Omega,{n}}^{(L_T,M_T)}:=
\widehat{f}_{\Omega,{n}}^{(L_T,M_T)}
-\int_Y \widehat{f}_{\Omega,{n}}^{(L_T,M_T)}\circ a(n)\, d\nu.
$$
Then 
$$
g_{\Omega,{n}}^{(L_T)}=g_{\Omega,{n}}^{(M_T)}+g_{\Omega,{n}}^{(L_T,M_T)}.
$$
We introduce the subsets
$$
B_{n,k}:=\left\{\Lambda\in Y:\,2^{k}\leq \textup{ht}(a(n)\Lambda)< 2^{k+1}\right\}.
$$
In view of \eqref{eq:f_uniform} and \eqref{etaL}, 
$$
\widehat f_{\Omega,{n}}^{(M_T)}\circ a({n})\ll_c  2^k\quad\;\hbox{on $B_{n,k}$}.
$$
Therefore, from Lemma \ref{lem:BnY}, we obtain that
\begin{align*}
\int_{Y} \widehat f_{\Omega,{n}}^{(M_T)}\circ a({n})\,d\nu&\ll_c \sum_{k:\, 2^{k+1}\ge M_T/2}2^{k}\cdot\nu(B_{n,k})\nonumber \\
&\ll \sum_{k:\, 2^{k+1}\ge M_T/2} 2^{k}\left(2^{-3k}+2^{-2k}\cdot e^{-\lfloor n\rfloor}\right)\nonumber\\
&\ll M_T^{-2}+M_T^{-1}\cdot e^{-\lfloor n\rfloor}.
\end{align*}
Then
\begin{align}\label{esss1}
\left\|\sum_{n\in\cF_\Omega}g_{\Omega,{n}}^{(M_T)}\circ a(n)\right\|_{L^1(\nu)}
&\le 2\sum_{n\in\cF_\Omega}  \int_Y \widehat{f}_{\Omega,{n}}^{(M_T)}\circ a({n})\,d\nu\\
&\ll M_T^{-2}(\ln T)^2+M_T^{-1}\ln T.\nonumber
\end{align}
This gives
$o(V_\Omega^{1/2})=o(b^{1/2}\ln T)$ when
we take $M_T=(\ln T)^{K_2'}$ with sufficiently large $K_2'$ due
to our assumption on $b$.

Now it remains to estimate the $L^1$-norm of 
$\sum_{n\in\cF_\Omega} g_{\Omega,{n}}^{(L_T,M_T)}\circ a(n)$.
We use that
$$
\left\|\sum_{n\in \cF_\Omega}g_{\Omega,{n}}^{(L_T,M_T)}\right\|_{L^1(\nu)}\le 
\left\|\sum_{n\in \cF_\Omega}g_{\Omega,{n}}^{(L_T,M_T)}\right\|_{L^2(\nu)}
$$
and apply Proposition \ref{l:var_general}  to the functions $g_{\Omega,{n}}^{(L_T,M_T)}$.
Since $0\le \eta_{L_T}\le 1$,
\begin{align*}
\left\|g_{\Omega,{n}}^{(L_T,M_T)}\circ a(n)\right\|_{{L}^2(\nu)}
&\le 2\left\|\widehat{f}_{\Omega, n}\circ a(n)\right\|_{{L}^2(\nu)},\\
\left\|g_{\Omega,{n}}^{(L_T,M_T)}\circ a({n})\right\|_{{L}^2(\nu_i)}&\leq \left\|\widehat{f}_{\Omega,{n}}\circ a({n})\right\|_{{L}^2(\nu_i)}+\left\|\widehat{f}_{\Omega, n}\circ a({n})\right\|_{{L}^1(\nu)}.
\end{align*}
Recalling \eqref{eq:bbound}, we deduce from Proposition \ref{prop:L2} that since $a\ge (\ln T)^{K_1-K_0}$,
\begin{equation}\label{eq:fff_n}
\left\|\widehat{f}_{\Omega, n}\circ a(n)\right\|_{{L}^2(\nu)}
\ll_c \max(1,\ln(b/a))\max(1,|n|)^{1/2}\ll_{K_0,K_1}(\ln\ln T) \max(1,|n|)^{1/2}.
\end{equation}
This verifies Condition (ii) of  Proposition \ref{l:var_general}.
Condition (iii) follow from Proposition \ref{prop:alphaintoverY} and 
the first part of Proposition \ref{prop:L1} using \eqref{eq:bbound}.
Using that $L_T,M_T\le (\ln T)^{K_2'}$ and $\varepsilon_T\ge (\ln T)^{-K_0}$,
we deduce that condition (iv) follows from \eqref{eq:norm_bound}. 
Now applying Proposition \ref{l:var_general}, we obtain that
for all $\delta>0$,
\begin{align}
\int_Y \left(\sum_{n\in \cF_\Omega}g_{\Omega,{n}}^{(L_T,M_T)}\circ a({n})\right)^2\, d\nu
= &\sum_{(n,k)\in \cF_\Omega^\triangledown} \int_X g_{\Omega,{n}}^{(L_T,M_T)}\circ a({n})\cdot
g_{\Omega,{k}}^{(L_T,M_T)}\circ a({k})\, d\mu \nonumber\\
&\;\;+O_{c,\delta}\left((1+(\ln T) b^{1-\delta}) (\ln T)^{\delta}\right),\label{eq:app}
\end{align}
where we additionally used that $a\ge (\ln T)^{-(K_0-K_1)}$.
 By the Cauchy--Schwarz inequality,
$$
\int_X g_{\Omega,{n}}^{(L_T,M_T)}\circ a({n})\cdot
g_{\Omega,{k}}^{(L_T,M_T)}\circ a({k})\, d\mu
\le \left\|g_{\Omega,{n}}^{(L_T,M_T)}\right\|_{L^2(\mu)}\cdot
\left\|g_{\Omega,{k}}^{(L_T,M_T)}\right\|_{L^2(\mu)},
$$
and 
$$
\left\|g_{\Omega,{n}}^{(L_T,M_T)}\right\|_{L^2(\mu)}
\le \left\|\widehat{f}_{\Omega,{n}}^{(L_T,M_T)}\right\|_{L^2(\mu)}+
\left\|\widehat{f}_{\Omega,{n}}^{(L_T,M_T)}\circ a(n)\right\|_{L^1(\nu)}.
$$
Since $\lfloor n\rfloor\geq B_1\ln\ln T$
when $(n,k)\in \cF_\Omega^\triangledown$, it follows from Theorem \ref{prop:ED} that 
$$
\int_{Y} \widehat{f}_{\Omega,{n}}^{(L_T,M_T)}\circ a({n})\,d\nu
=\int_{X}\widehat{f}_{\Omega,{n}}^{(L_T,M_T)}\,d\mu
+O\left(\left\|\widehat{f}_{\Omega,{n}}^{(L_T,M_T)}\right\|_{C^\ell} (\ln T)^{-\gamma_1 B_1} \right).
$$
Further, since $L_T,M_T\le (\ln T)^{K_2'}$ and $\varepsilon_T\ge (\ln T)^{-K_0}$, it follows from  \eqref{eq:norm_bound} that
$$
\left\|\widehat{f}_{\Omega,{n}}^{(L_T,M_T)}\right\|_{C^\ell}\ll_\ell (\ln T)^{\sigma-\gamma_1 B_1}
$$
for some $\sigma=\sigma(K_0,K_2')>0$.
Therefore, we conclude that 
$$
\left\|g_{\Omega,{n}}^{(L_T,M_T)}\right\|_{L^2(\mu)}\ll 
\left\|\widehat{f}_{\Omega,{n}}^{(L_T,M_T)}\right\|_{L^2(\mu)}+
(\ln T)^{\sigma-\gamma_1 B_1}.
$$
Using Lemma \ref{l:integrable} and 
Proposition \ref{prop:alphaintoverX},
we deduce from H\"older's Inequality that
for $p<3/2$
\begin{align*}
\left\|\widehat{f}_{\Omega,{n}}^{(L_T,M_T)}\right\|_{L^2(\mu)}^2
&\le \int_{\{\hbox{\tiny ht} \ge L_T/2\}} \widehat{f}_{\Omega,{n}}^2\, d\mu
\le \left( \int_X \widehat{f}_{\Omega,{n}}^{2p}\, d\mu\right)^{1/p}
\left( \int_X \chi_{\{\hbox{\tiny ht} \ge L_T/2\}}\, d\mu\right)^{(p-1)/p}\\
&\ll_{c,p,\delta} b^{2/3} L_T^{-(3-\delta)(p-1)/p}.
\end{align*}
This implies that for any $\delta>0$,
\begin{align*}
\left\|g_{\Omega,{n}}^{(L_T,M_T)}\right\|_{L^2(\mu)}
\ll_{c,\delta} b^{1/3} L_T^{-1/2+\delta}+(\ln T)^{\sigma-\gamma_1 B_1}.
\end{align*}
Therefore, since 
\begin{equation}\label{eq:FFF_o}
|\cF_\Omega^\triangledown|\ll_B |\cF_\Omega|(\ln\ln T)^2\ll_c  (\ln T)^2(\ln\ln T)^2,
\end{equation}
we obtain from \eqref{eq:app} that
\begin{align*}
\left\|\sum_{n\in \cF_\Omega}g_{\Omega,{n}}^{(L_T,M_T)}\circ a({n}) \right\|_{L^1(\nu)} \ll_{c,\delta,B} & \Big(\big(b^{1/3} L_T^{-1/2+\delta}+(\ln T)^{\sigma-\gamma_1 B_1}\big)^2  (\ln T)^2(\ln\ln T)^2 \\
&\quad\quad\quad+ (1+(\ln T) b^{1-\delta}) (\ln T)^{\delta}\Big)^{1/2}
\end{align*}
for any $\delta>0$.
We need to ensure that this bound is $o(V_\Omega^{1/2})=o(b^{1/2}\ln T)$, which follows
by taking $B_1$ sufficiently large and using that 
$L_T\ge b^{-1/3}(\ln T)^{\theta}$ and $b\ge (\ln T)^{-2+\theta}$.
\end{proof}

\section{Estimating the variance: second step} \label{sec:CompVar}

In  this section we analyze the variance with respect to the measure $\nu$ of the sums $\Phi_\Omega$ defined in \eqref{eq:sums}.
We assume that for some $K_0,K_1,K_2>0$ and $\theta>0$,
\begin{align}\label{eq:var_cond}
&a\le b/2,\;\;
(\ln T)^{-K_0}\leq \eps_T \le a (\ln T)^{-K_1}, \\
&b^{-1/3}(\ln T)^\theta\leq L_T \leq (\ln T)^{K_2}, \;\;
b\ge (\ln T)^{-\min(2,K_1)+\theta}.\nonumber
\end{align}
Throughout this section, we use that
$$
|\cF_\Omega|\ll_c (\ln T)^2\qand V_\Omega\gg (\ln T)^2 b,
$$
which follows from \eqref{eq:cardQRT} and \eqref{eq;voll},
since $a\le b/2$ and $b\ge (\ln T)^{-2+\theta}$.

We apply Proposition \ref{l:var_general} to the functions $\varphi_{\Omega,n}$ (cf. \eqref{def_varphiTn}).
Since $0\le \eta_{L_T}\le 1$,
\begin{align*}
\big\|\varphi_{\Omega, n}\circ a(n)\big\|_{{L}^2(\nu)}
&\le 2\left\|\widehat{f}_{\Omega, n}\circ a(n)\right\|_{{L}^2(\nu)},\\
\big\|\varphi_{\Omega,{n}}\circ a({n})\|_{{L}^p(\nu_i)} &\leq \left\|\widehat{f}_{\Omega, n}\circ a({n})\right\|_{{L}^p(\nu_i)}+\left\|\widehat{f}_{\Omega, n}\circ a({n})\right\|_{{L}^1(\nu)}.   
\end{align*}
Thus, Condition (ii) follows from \eqref{eq:fff_n},
and Condition (iii) follows from Proposition \ref{prop:alphaintoverY}
and the first part of Propositions \ref{prop:L1}.
In view of our assumptions on $\varepsilon_T$ and $L_T$,
Condition (iv) follows from \eqref{eq:norm_bound}. 
Applying Proposition \ref{l:var_general} and noting that $b\ge (\ln T)^{-2+\theta}$, we conclude that 
\begin{align*}
\int_Y\Phi_\Omega^2\,d\nu &=\sum_{(n,k)\in \cF_\Omega^\triangledown} \int_X 
\varphi_{\Omega, n}\circ a(n)\cdot \varphi_{\Omega, k}\circ a(k)\, d\mu +
O_{c,\delta}\Big((1+(\ln T)b^{1-\delta})(\ln T)^\delta\Big)
\end{align*}
for all $\delta>0$, so that 
\begin{align*}
\int_Y\Phi_\Omega^2\,d\nu&=\sum_{(n,k)\in \cF_\Omega^\triangledown} \int_X 
\varphi_{\Omega, n}\circ a(n)\cdot \varphi_{\Omega, k}\circ a(k)\, d\mu +
o(V_\Omega).
\end{align*}
Next, we show that
$$
\sum_{(n,k)\in \cF_\Omega^\triangledown} \int_X 
\varphi_{\Omega, n}\circ a(n)\cdot \varphi_{\Omega, k}\circ a(k)\, d\mu=
\sum_{(n,k)\in \cF_\Omega^\triangledown} \int_X 
\psi_{\Omega, n}\circ a(n)\cdot \psi_{\Omega, k}\circ a(k)\, d\mu +
o(V_\Omega).
$$
We have
\begin{align}\label{eq:long}
\int_{X}\varphi_{\Omega, n}\circ a({n}) \cdot \varphi_{\Omega, k}\circ a(k)\,d\mu
=&
\int_{X}\psi_{\Omega,n}\circ a({n}) \cdot \psi_{\Omega, k}\circ a(k)\,d\mu \\
&+\int_{X}(\varphi_{\Omega, n}-\psi_{\Omega, n})\circ a({n}) \cdot \varphi_{\Omega, k}\circ a(k)\,d\mu \nonumber\\
&+\int_{X}\psi_{\Omega, n}\circ a({n}) \cdot (\varphi_{\Omega, k}- \psi_{\Omega, k})\circ a(k)\,d\mu. \nonumber
\end{align}
By the Cauchy--Schwartz inequality,
$$
\left|\int_{X}(\varphi_{\Omega, n}-\psi_{\Omega, n})\circ a({n}) \cdot \varphi_{\Omega, k}\circ a(k)\,d\mu
\right|\le 
\left\|\varphi_{\Omega, n}-\psi_{\Omega, n}\right\|_{L^2(\mu)}
\cdot 
\left\|\varphi_{\Omega, k}\right\|_{L^2(\mu)}.
$$
Since $0\le \eta_{L_T}\le 1$, it follows from Lemma \ref{l:f_l1_l2} and Proposition \ref{prop:L1} that
$$
\left\|\varphi_{\Omega, k}\right\|_{L^2(\mu)}\le 
\left\|\widehat{f}_{\Omega, k}\right\|_{L^2(\mu)}+
\left\|\widehat{f}_{\Omega, k}\right\|_{L^1(\nu)}\ll_c b^{1/2}.
$$
The other factor is estimated as 
\begin{align*}
\left\|\varphi_{\Omega, n}-\psi_{\Omega, n}\right\|_{L^2(\mu)}\le& 
\left\|\widehat f_{\Omega,{n}}\cdot\eta_{L_T}-\widehat\chi_{\Delta_{\Omega,{n}}}\right\|_{{L}^2(\mu)}\\
&+
\left|\int_{Y}\left(\widehat{f}_{\Omega,{n}}\cdot \eta_{L_T}\right)\circ a({n})\,d\nu-\int_{X}\widehat\chi_{\Delta_{\Omega,{n}}}\,d\mu\right|.
\end{align*}
Since $\lfloor n\rfloor\geq B_1\ln\ln T$, it follows from Theorem \ref{prop:ED} that 
\begin{equation}\label{eq:conv0}
\int_{Y}\left(\widehat{f}_{\Omega,{n}}\cdot \eta_{L_T}\right)\circ a({n})\,d\nu
=\int_{X}\widehat{f}_{\Omega,{n}}\cdot \eta_{L_T}\,d\mu
+O\left(\left\|\widehat{f}_{\Omega,{n}}\cdot \eta_{L_T}\right\|_{C^\ell} (\ln T)^{-\gamma_1 B_1} \right).
\end{equation}
Since $L_T\le (\ln T)^{K_2}$ and $\varepsilon_T\ge (\ln T)^{-K_0}$,
it follows from  (\ref{eq:norm_bound}) that
$$
\left\|\widehat{f}_{\Omega,{n}}\cdot \eta_{L_T}\right\|_{C^\ell}\ll_\ell (\ln T)^{\sigma-\gamma_1 B_1}
$$
for some $\sigma=\sigma(K_0,K_2)>0$.
Therefore, the sum coming from the second terms in \eqref{eq:conv0} gives a negligible contribution when $B_1$ is chosen sufficiently large. Further,
$$
\left|\int_{X}\widehat{f}_{\Omega,{n}}\cdot \eta_{L_T}\,d\mu-\int_{X}\widehat\chi_{\Delta_{\Omega,{n}}}\,d\mu\right|\le
\left\|\widehat f_{\Omega,{n}}\cdot\eta_{L_T}-\widehat\chi_{\Delta_{\Omega,{n}}}\right\|_{{L}^2(\mu)},
$$
so that it remains to deal with these $L^2$-norms, which we estimate as
$$
\left\|\widehat{f}_{\Omega,{n}}\cdot \eta_{L_T}- \chi_{\Delta_{\Omega,{n}}}\right\|_{L^2(\mu)}
\le \left\|\widehat f_{\Omega, {n}}-\widehat\chi_{\Delta_{\Omega,{n}}}\right\|_{{L}^2(\mu)}+\left\|(1-\eta_{L_T})\cdot \widehat\chi_{\Delta_{\Omega,{n}}}\right\|_{{L}^2(\mu)}
$$
using that $0\le \eta_T\le 1$.
According to Lemma \ref{p:app_l2},
$$
\left\|\widehat f_{\Omega, {n}}-\widehat\chi_{\Delta_{\Omega,{n}}}\right\|_{{L}^2(\mu)}\ll_c
\max\left(\varepsilon_T,-\frac{\varepsilon_T}{a}\ln\left(\frac{\varepsilon_T}{a}\right)\right)^{1/2}.
$$
To estimate the last term, we apply the H\"older inequality with 
$1/p+1/q=1/2$, $q< 3$:
$$
\left\|(1-\eta_{L_T})\cdot \widehat\chi_{\Delta_{\Omega,{n}}}\right\|_{{L}^2(\mu)}\le \mu\Big(\{\height\ge L_T/2 \}\Big)
^{1/p} \left\|\widehat\chi_{\Delta_{\Omega,{n}}}\right\|_{{L}^q(\mu)}
\ll_{c,q} L_T^{-q/p}b^{1/3},
$$
where we used Lemma \ref{l:integrable} and Proposition \ref{prop:alphaintoverX}.

Combining the above estimates, we deduce that 
\begin{align*}
\int_{X}(\varphi_{\Omega, n}-\psi_{\Omega, n})\circ a({n}) &\cdot \varphi_{\Omega, k}\circ a(k)\,d\mu \\
=&\,
O_{c,q}\left(
b^{1/2}
\max\left(\varepsilon_T,-\frac{\varepsilon_T}{a}\ln\left(\frac{\varepsilon_T}{a}\right)\right)^{1/2}
+L_T^{-q/p}b^{5/6}
\right)\\
=&\,
O_{c,K_1,\delta}\left( b^{1/2} (\ln T)^{-K_1/2} (\ln\ln T)^{1/2}+ L_T^{-1/2+\delta}b^{5/6}
\right)
\end{align*}
for any $\delta>0$.
The last term in (\ref{eq:long}) can be estimated similarly. Using \eqref{eq:FFF_o}, we conclude that 
\begin{align*}
&\sum_{(n,k)\in \cF_\Omega^\triangledown} \int_X 
\varphi_{\Omega, n}\circ a(n)\cdot \varphi_{\Omega, k}\circ a(k)\, d\mu\\
= & \sum_{(n,k)\in\cF_\Omega^\triangledown}
\int_{X}\psi_{\Omega, n}\circ a({n}) \cdot \psi_{\Omega, k}\circ a(k)\,d\mu\\
 &\quad\quad + O_{c,B,K_1,\delta}\left( b^{1/2} (\ln T)^{2-K_1/2} (\ln\ln T)^{5/2}+ (\ln T)^{2} (\ln\ln T)^2 L_T^{-1/2+\delta} b^{5/6}\right)
\end{align*}
for every $\delta>0$.
Since
$L_T\ge b^{-1/3}(\ln T)^\theta$ and $b\ge (\ln T)^{-K_1+\theta}$, the last term is $o(V_\Omega)$ as $T\to\infty$.
Therefore, we conclude that 
$$
\int_{Y}\Phi_\Omega^{2}\, d\nu 
={\sum}_{(n,k)\in\cF_\Omega^\triangledown}
\int_{X}\psi_{\Omega, n}\circ a({n}) \cdot \psi_{\Omega, k}\circ a(k)\,d\mu
+o(V_\Omega)\quad\hbox{as $T\to \infty$.}
$$
Recalling \eqref{eq:phiTRn} and  using Theorem \ref{thm:Rogers}, we obtain 
\begin{align*}
\int_{X}\psi_{\Omega, n} \cdot \psi_{\Omega, k}\,d\mu
&=
\int_{X}\widehat\chi_{a(-n)\Delta_{\Omega,{n}}}\cdot\widehat\chi_{a(-k)\Delta_{\Omega, {n}}}\,d\mu-\Vol(\Delta_{\Omega,{n}})\cdot\Vol(\Delta_{\Omega,k})\\  
&= \zeta(3)^{-1} \sum_{p,q\in\bN} \Vol\left(p^{-1} \cdot a(-n)\Delta_{\Omega,n}\cap q^{-1}\cdot a(-k)\Delta_{\Omega,k}\right).
\end{align*}

The above computation verifies the following lemma:

\begin{lem}
\label{lem:variance}
Under the assumptions \eqref{eq:var_cond},
there exist constants $B,B_0,B_1>0$, depending only on $K_0,K_1,K_2$ such that
$$
\int_{Y}\Phi_\Omega^{2}\, d\nu 
=\zeta(3)^{-1}\sum_{(n,k)\in\cF_\Omega^\triangledown}
\sum_{p,q\in\bN} V_{p,q}(n,k)
+o(V_\Omega)\quad\hbox{as $T\to\infty$},
$$
where $\cF_\Omega^\triangledown=\cF_\Omega^\triangledown(B,B_0,B_1)$ is defined in \eqref{eq:ftri} and
$$
V_{p,q}(n,k):=\Vol\left(p^{-1} \cdot a(-n)\Delta_{\Omega,{n}}\cap q^{-1}\cdot a(-k)\Delta_{\Omega,k}\right).
$$
\end{lem}

The following lemma complements the above result:

\begin{lem}\label{lem:variancepositive}
Assume that for some $K_4>0$ and $\varepsilon>0$.
\begin{equation}\label{eq:var_cond333}
(\ln T)^{-K_4}\leq a\leq b (\ln \ln T)^{-4-\varepsilon}.
\end{equation}
Then
$$
V_\Omega^{-1} \cdot \sum_{(n,k)\in\cF_\Omega^\triangledown}
\sum_{p,q\in\bN} V_{p,q}(n,k) \longrightarrow 
\sum_{m\in\mb Z^2} \sum_{p,q\in\bN} V_{p,q}(m)\quad\hbox{ as $T\to\infty$,}
$$
where $\cF_\Omega^\triangledown$ is defined in \eqref{eq:ftri} 
and
$$
V_{p,q}(m):=\Vol\left(p^{-1}\cdot a(m) \Delta_c\cap q^{-1}\cdot \Delta_c\right)
$$
with $\Delta_c$ defined in \eqref{eq:Deltac}.
\end{lem}

\begin{proof}
We note that from \eqref{eq;voll} and our assumption \eqref{eq:var_cond333},
\begin{equation}\label{eq:FFF_o1}
V_\Omega\sim (\ln T)^2 b\quad\hbox{as $T\to\infty$}
\end{equation}
and from \eqref{eq:subb}, 
$$
a(-m)\Delta_{\Omega,m}\subset \big(e^{-m_1-1}c,e^{-m_1}c\big]\times \big(e^{-m_2-1}c,e^{-m_2}c\big]\times \big(e^{m_1+m_2}ac^{-2},e^{m_1+m_2+2}bc^{-2}\big].
$$
Let $K_{p,q}(n):=\{k:\, V_{p,q}(n,k)\ne 0\}$.
Then when $k\in K_{p,q}(n)$,
\begin{equation}
\label{eq:nonemptyint}
\frac{e^{-n_i-1}}{p}\leq \frac{e^{-k_i}}{q}\qand \frac{e^{-k_i-1}}{q}\leq \frac{e^{-n_i}}{p}\quad\quad \hbox{for $i=1,2$.}   
\end{equation}
This gives at most three possibilities for $k_i$ 
when the other parameters are fixed, and it follows that 
$|{K_{p,q}(n)}|\le  9.$
Therefore, it follows from the Cauchy--Schwarz inequality that 
\begin{align}\label{eq:prev}
&\sum_{(n,k)\in\cF_\Omega^\triangledown} \sum_{p,q\in\bN} V_{p,q}(n,k)\\    
\le & \, \sum_{n\in \cF_\Omega} \sum_{p,q\in\bN}
\sum_{k\in K_{p,q}(n)} 
\Vol\left(p^{-1} \cdot a(-n)\Delta_{\Omega,n}\right)^{1/2}\cdot \Vol\left( q^{-1}\cdot a(-k)\Delta_{\Omega,k}\right)^{1/2} \nonumber\\
\le & \, \sum_{n\in \cF_\Omega} \sum_{p,q\in\bN}
\sum_{k\in K_{p,q}(n)} 
p^{-3/2}\Vol\left(\Delta_{\Omega,n}\right)^{1/2}\cdot q^{-3/2}\Vol\left(\Delta_{\Omega,k}\right)^{1/2} 
\ll_c \, |\cF_\Omega| b \ll V_\Omega. \nonumber
\end{align}
We refine this computation to show that the normalized above sum converges.
Let us introduce the following super-sets of 
$\Delta_{\Omega,{n}}$:
\begin{equation}
\label{eq:DeltaRTntilde0}
\Delta_{\Omega} := \left\{ (x_1,x_2,y)\in \bR^3 \,  : \,  
\begin{array}{c}
a < |x_1 x_2| \, y \leq b \\[0.3cm]
e^{-1}c < |x_1|, |x_2| \leq c 
\end{array} 
\right\}\subset [-c,c]^2\times [ac^{-2},e^2bc^{-2}].
\end{equation}
We observe that $\Delta_{\Omega}=\Delta_{\Omega,{n}}$
unless
\begin{equation}\label{eq:ccond0}
e^2 bc^{-2}\ge T e^{-(n_1+n_2)}\quad\hbox{or}\quad ac^{-2}\le e^{-(n_1+n_2)}.  
\end{equation}
The first condition in \eqref{eq:ccond0} gives a subset of $n\in\mathcal{F}_\Omega$ with
$$
\ln\big(T e^{-2}b^{-1}c^2\big)\le n_1+n_2<\beta_\Omega=\ln\big(Ta^{-1}c^2\big),
$$
and the number of such $n$ is estimated as
\begin{equation}\label{eq:t1}
\le \left(\ln\big(Ta^{-1}c^2\big)+1\right)\cdot \left(\ln\big(e^2ba^{-1}\big)+1\right) \ll_{K_4} \ln T\cdot \ln\ln T.
\end{equation}
Similarly, the second condition in \eqref{eq:ccond0} gives a subset of $n\in\mathcal{F}_\Omega$ with
$$
\alpha_\Omega=\ln\big(e^{-2}b^{-1}c^2\big)\le n_1+n_2\le\ln\big(a^{-1}c^2\big),
$$
and the number of such $n$ is estimated as
\begin{equation}\label{eq:t2}
\le \left(\ln\big(a^{-1}c^2\big)+1\right) \left(\ln\big(e^2ba^{-1}\big)+1\right)
\ll_{K_4} (\ln\ln T)^2.
\end{equation}
Let $\cF_\Omega^\diamond$ denote the subset of $(n,k)\in \cF_\Omega^\triangledown$
where both $n$ and $k$ do not satisfy \eqref{eq:ccond0}.
Then arguing as in \eqref{eq:prev}, we conclude that
$$
\sum_{(n,k)\in\cF_\Omega^\triangledown\setminus \cF_\Omega^\diamond} \sum_{p,q\in\bN} V_{p,q}(n,k)\ll_c \ln T\cdot \ln\ln T\cdot b =o(V_\Omega).
$$
Therefore, it remains to analyze the sum 
$\sum_{(n,k)\in\cF_\Omega^\diamond} \sum_{p,q\in\bN} V'_{p,q}(n,k),$
where
$$
V'_{p,q}(n,k):=\Vol\left(p^{-1} \cdot a(-n)\Delta_{\Omega}\cap q^{-1}\cdot a(-k)\Delta_{\Omega}\right).
$$
Further, we consider the sets
\begin{equation}
\label{eq:DeltaRTntilde}
\widetilde\Delta_{\Omega} := \left\{ (x_1,x_2,y) \,  : \,  
\begin{array}{c}
0 < |x_1 x_2| \cdot y \leq b \\[0.3cm]
e^{-1}c < |x_1|, |x_2| \leq c 
\end{array} 
\right\}
\end{equation}
and define
\begin{align*}
\widetilde V_{p,q}(n,k):=&\Vol\left(p^{-1} \cdot a(-n)\widetilde\Delta_{\Omega}\cap q^{-1}\cdot a(-k)\widetilde\Delta_{\Omega}\right)\\
=&\,b\cdot \Vol\left(p^{-1} \cdot a(-n)\Delta_c\cap q^{-1}\cdot a(-k)\Delta_c\right)=b\cdot V_{p,q}(n-k).
\end{align*}
It follows from our assumption \eqref{eq:var_cond333} that
$$
\Vol\left(\widetilde\Delta _{\Omega}\setminus \Delta_{\Omega}\right)\ll_{c} a\le b\cdot (\ln\ln T)^{-4-\varepsilon}.
$$
We shall use the following estimate: for subsets $\Delta_1\subset \widetilde\Delta_1$ and $\Delta_2\subset\widetilde\Delta_2$ of $\bR^3$,
\begin{align*}
\Vol(\widetilde\Delta_1\cap \widetilde\Delta_2)-\Vol(\Delta_1\cap\Delta_2)
\le  &\Vol(\widetilde\Delta_1\setminus \Delta_1)^{1/2}\Vol(\widetilde\Delta_2)^{1/2}\\
&\quad+\Vol(\Delta_1)^{1/2} \Vol(\widetilde\Delta_2\setminus \Delta_2)^{1/2},
\end{align*}
which follows from the Cauchy--Schwarz Inequality.
This gives
$$
\widetilde V_{p,q}(n,k)=V_{p,q}'(n,k)+O_c\left(p^{-3/2}q^{-3/2}\cdot b\cdot (\ln\ln T)^{-2-\varepsilon/2}\right),
$$
so that in view of \eqref{eq:FFF_o} and \eqref{eq:FFF_o1}
$$
\sum_{(n,k)\in\cF_\Omega^\diamond} \sum_{p,q\in\bN} \widetilde V_{p,q}(n,k)
=\sum_{(n,k)\in\cF_\Omega^\diamond} \sum_{p,q\in\bN} V'_{p,q}(n,k)
+o(V_\Omega)\quad\hbox{as $T\to\infty$.}
$$
Now it remains to compute the limit
$$
\lim_{T\to\infty} (\ln T)^{-2} \sum_{(n,k)\in\cF_\Omega^\diamond} \sum_{p,q\in\bN} V_{p,q}(n-k).
$$
It follows from the Cauchy--Schwarz inequality that 
\begin{equation}\label{eq:vpq}
V_{p,q}(m)\ll_c p^{-3/2} q^{-3/2}\quad\hbox{uniformly on $m\in\bZ^2$.}
\end{equation}
We recall that pairs $(n,k)\in \cF_\Omega^\triangledown$ satisfy 
$|n-k|< B\, \ln\ln T$, so that taking \eqref{eq:t1} and \eqref{eq:t2} into account,
we obtain that
$$
|\cF_\Omega^\triangledown\setminus\cF_\Omega^\diamond|\ll_{B,K_4} \ln T\cdot (\ln\ln T)^3. 
$$
Therefore, because of \eqref{eq:vpq},
$$
\sum_{(n,k)\in\cF_\Omega^\diamond} \sum_{p,q\in\bN} V_{p,q}(n-k)
=\sum_{(n,k)\in\cF_\Omega^\triangledown} \sum_{p,q\in\bN} V_{p,q}(n-k)+o\big((\ln T)^2\big).
$$
Let $\cF_\Omega^\square$ consist of $(n,k)\in \cF_\Omega\times \cF_\Omega$ such that 
$|n-k|< B\, \ln\ln T$. We note that the number of pairs
$(n,k)\in \cF_\Omega^\square$ such that either
$|n|< B_0\ln\ln T$ or $\lfloor n \rfloor <  {B_1}\ln\ln T$
is $O_{B,B_0,B_1,c}\big(\ln T (\ln\ln T)^3\big)$.
Thus, because of \eqref{eq:vpq},
$$
\sum_{(n,k)\in\cF_\Omega^\triangledown} \sum_{p,q\in\bN} V_{p,q}(n-k)
=\sum_{(n,k)\in\cF_\Omega^\square} \sum_{p,q\in\bN} V_{p,q}(n-k)+o\big((\ln T)^2\big).
$$
We observe that when $V_{p,q}(n-k)\ne 0$,
$$p^{-1}\cdot a(-n) \Delta_c\cap q^{-1}\cdot a(-k)\Delta_c\ne \emptyset$$
and
$pe^{n_i}\asymp_c qe^{k_i}$ for $i=1,2$, so that $\max(p,q)\gg_c e^{|n-k|}$. Hence, using \eqref{eq:vpq},
\begin{align*}
\sum_{(n,k)\in(\cF_\Omega\times\cF_\Omega) \setminus\cF_\Omega^\square} \sum_{p,q\in\bN} V_{p,q}(n-k) 
&\ll_c |\cF_\Omega|^2 \sum_{p,q\in\bN:\, \max(p,q)\gg_c (\ln T)^{B}}p^{-3/2} q^{-3/2}\\
&=o\big((\ln T)^2\big)
\end{align*}
when $B$ is sufficiently large.
Finally,
\begin{align*}
(\ln T)^{-2}\sum_{n,k\in\cF_\Omega} \sum_{p,q\in\bN} V_{p,q}(n-k)
=\sum_{m\in \bZ^2} \sum_{p,q\in\bN} \frac{|\cF_\Omega\cap (m+\cF_\Omega)|}{(\ln T)^2}\cdot V_{p,q}(m).
\end{align*}
We observe that when $V_{p,q}(m)\ne 0$, we have $p e^{m_i}\asymp_c q$ for $i=1,2$, and it follows that $\max(p,q)\gg_c e^{|m|}$. Using \eqref{eq:vpq}, we conclude that
$$
\sum_{m\in \bZ^2} \sum_{p,q\in\bN}  V_{p,q}(m)
\ll_c \sum_{m\in \bZ^2} e^{-|m|/2}<\infty.
$$
Since
$$
|\cF_\Omega|\ll_c (\ln T)^2\qand 
|\cF_\Omega\cap (m+\cF_\Omega)|\sim (\ln T)^2\;\;\hbox{as $T\to\infty$},
$$
it follows from the Dominated Convergence Theorem that 
\begin{align*}
(\ln T)^{-2} \sum_{n,k\in\cF_\Omega} \sum_{p,q\in\bN} V_{p,q}(n-k)
\longrightarrow \sum_{m\in \bZ^2} {\sum}_{p,q\in\bN}  V_{p,q}(m)\quad\hbox{as $T\to\infty$},
\end{align*}
which completes the proof of the lemma.
\end{proof}

Combining Lemmas \ref{lem:variance} and \ref{lem:variancepositive}, we obtain the following

\begin{prop}\label{th:variance}
Under the assumptions \eqref{eq:var_cond} and \eqref{eq:var_cond333},
$$
V_\Omega^{-1}\cdot \int_{Y}\Phi_\Omega^{2}\, d\nu \to 
\zeta(3)^{-1} \sum_{m\in\mb Z^2} \sum_{p,q\in\bN} V_{p,q}(m)\quad\hbox{ as $T\to\infty$.}
$$
\end{prop}

\section{Estimating higher-order cumulants}
\label{sec:higherord}

The goal of this section is to analyze the higher moments 
or equivalently cumulants of the functions $\Phi_\Omega$ defines in \eqref{eq:sums}. 
Before stating the main result, let us first introduce relevant notations.
Given bounded measurable functions $\phi_i:X\to \mb R$, $i=1,\dotsc,r$,
the cumulant with respect to the measure $\nu$ is defined as
\begin{equation} \label{eq:defcum}    
\textup{Cum}(\phi_{1},\dotsc,\phi_{r}):=\sum_{\mathcal{P}}(-1)^{|\mathcal{{P}}|-1}(|\mc{P}|-1)!\prod_{I\in\mc{P}}\int_{Y}{\prod}_{i\in I}\phi_{i}\,d\nu,
\end{equation}
where the sum ranges over all partitions $\mc{P}$ of the set $\{1,\dotsc,r\}$. Further,  for a single bounded function  $\phi$, one defines
$$
\textup{Cum}_r(\phi):=
\textup{Cum}(\phi,\dotsc,\phi).
$$
The main result is the following:

\begin{prop} \label{prop:higherlimit}
Assume that for some $K_0>0$ and $\theta>0$,
\begin{equation}\label{eq:high_cond}
a\le b/2,\quad (\ln T)^{-K_0}\leq \eps_T,\quad
L_T\le V_{\Omega}^{1/2},\quad 
b\ge (\ln T)^{-2+\theta}.
\end{equation}
Then for all $r\geq 3$, 
$$
\textup{Cum}_{r}\left(\Phi_\Omega\right)=o(V_\Omega^{r/2})\quad\hbox{as $T\to\infty$.}
$$
\end{prop}

Let $\mc{Q}$ be a partition of $\{1,\dotsc,r\}$.
For $i,j\in\{1,\dotsc,r\}$ we write $i\sim_{\mc{Q}} j$ to indicate that $i$ and $j$ lie in the same atom of the partition $\mc{Q}$. 
For $0\le \alpha<\beta$, we define
\begin{align*}
\Delta_{\mc{Q}}(\alpha,\beta)&:=\left\{(n_{1},\dotsc,n_{r})\in (\mb{N}_o^{2})^{r}: 
\begin{tabular}{ll} 
$|n_{i}-n_{j}|\le \alpha$ & if $i\sim_{\mc{Q}}j$,\\
$|n_{i}-n_{j}|>\beta$ & if $i\not\sim_{\mc{Q}}j$
\end{tabular}
\right\},\\
\Delta(\alpha) &:=\left\{(n_{1},\dotsc,n_{r})\in (\mb{N}_o^{2})^{r}: | n_{i}-n_{j}|\le \alpha\mbox{ for all }i,j\right\}.
\end{align*}
Let us choose parameters
$$
0=\beta_0<\beta_1<3\beta_1<\beta_2<\cdots<\beta_{r-1}<3\beta_{r-1}<\beta_r.
$$
Then according to \cite[Prop. 6.2]{BGmult}, we have the following decomposition:
$$
(\mb{N}_o^{2})^{r}=\Delta(\beta_r)\cup\bigcup_{i=0}^{r-1}\bigcup_{|\mc{Q}|\geq 2}\Delta_{\mc{Q}}(3\beta_i,\beta_{i+1}),
$$
where the union is taken over all non-trivial partitions $\mc{Q}$ of 
$\{1,\ldots,r\}$. Using this decomposition, we obtain that
\begin{align}\label{eq:decomp0}
\textup{Cum}_{r}\left(\Phi_\Omega\right)
=&\sum_{({n}_{1},\dotsc,{n}_{r})\in \cF_\Omega^r}\textup{Cum}\big(\varphi_{\Omega,{n}_1}\circ a(n_1),\dotsc,\varphi_{\Omega,{n}_r}\circ a(n_r)\big)\\
\le &\sum_{({n}_{1},\dotsc,{n}_{r})\in \cF_\Omega^r\cap \Delta(\beta_r)} \left|\textup{Cum}\big(\varphi_{\Omega,{n}_1}\circ a(n_1),\dotsc,\varphi_{\Omega,{n}_r}\circ a(n_r)\big)\right|
 \nonumber \\
&+\sum_{j=0}^{r-1}\sum_{|\mc{Q}|\ge 2}\sum_{({n}_{1},\dotsc,{n}_{r})\in \cF_\Omega^r\cap \Delta_{\mc{Q}}(3\beta_j,\beta_{j+1})}\left|\textup{Cum}\big(\varphi_{\Omega,{n}_1}\circ a(n_1),\dotsc,\varphi_{\Omega,{n}_r}\circ a(n_r)\big)\right|. \nonumber
\end{align}
We analyze each of the sums separately in the following two lemmas.

\begin{lem}
\label{lem:cluster}
Under the assumptions \eqref{eq:high_cond}, 
when $\beta_r\ll \ln\ln T$,
\begin{align*}
\sum_{({n}_{1},\dotsc,{n}_{r})\in \cF_\Omega^{r}\cap\Delta(\beta_r)}
\Big|\textup{Cum}\big(\varphi_{\Omega,{n}_1}\circ a( n_1),\dotsc,\varphi_{\Omega,{n}_r}\circ a( n_r)\big)\Big|
=o\left(V_\Omega^{r/2}\right)\quad\hbox{as $T\to\infty$.}
\end{align*}
\end{lem}

\begin{proof}
Let us set
$$
\phi_{\Omega,n}:=\widehat f_{\Omega,n}\cdot \eta_{L_T}.
$$
We note that by \eqref{eq:f_uniform} and \eqref{etaL},
\begin{equation}\label{eq:LTT}
\phi_{\Omega,n}\ll_c L_T.
\end{equation}
We recall (cf. \eqref{def_varphiTn}) that $\phi_{\Omega,n}$ differs from
$\varphi_{\Omega,n}$ by a constant. Therefore, it follows from 
the properties of cumulants that
$$
\textup{Cum}\big(\varphi_{\Omega,{n}_1}\circ a( n_1),\dotsc,\varphi_{\Omega,{n}_r}\circ a( n_r)\big)
=
\textup{Cum}\big(\phi_{\Omega,{n}_1}\circ a( n_1),\dotsc,\phi_{\Omega,{n}_r}\circ a( n_r)\big).
$$
Now it remains to estimate the products 
$\prod_{I\in\mc{Q}}\int_{Y}\prod_{i\in I}\phi_{\Omega,{n}_{i}}\circ a(n_i)\,d\nu$
for partitions $\mc{Q}$ of $\{1,\ldots, r\}$.
Since $({n}_{1},\dotsc,{n}_{r})\in \Delta(\beta_r)$, for some $K_1>0$, 
\begin{equation}\label{eq:group}
|n_i-n_{j}|\le K_1\, \ln\ln T\quad\hbox{for all $i,j$.}
\end{equation}
We distinguish several possibilities.  
First, we analyze the case when
\begin{equation}\label{ass:1}
\lfloor n_{1}\rfloor\geq B_1\,\ln\ln T
\end{equation}
for a constant $B_1$ to be specified later.
We pick $i_0\in I$ and write
$$
F:=\prod_{i\in I}\phi_{\Omega,{n}_{i}}\circ a(n_i-n_{i_0}).
$$
From \eqref{ass:1} and \eqref{eq:group},
\begin{equation}\label{eq:ass1}
\lfloor n_{i}\rfloor\geq (B_1-K_1)\,\ln\ln T\quad\hbox{for all $i$.}
\end{equation}
Also in view of \eqref{eq:group}, it follows from \eqref{eq:norm_bound} and our assumptions on $L_T$ and $\varepsilon_T$ that 
$$
\|F\|_{C^\ell}\ll \prod_{i\in I}\big\|\phi_{\Omega,{n}_{i}}\circ a(n_i-n_{i_0})\big\|_{C^\ell}\ll (\ln T)^{K_2} 
$$
for some $K_2>0$ depending on $K_0$ and $K_1$. Therefore, it follows from Theorem \ref{prop:ED}
that
\begin{align}\label{eq:limm}
\int_{Y}\prod_{i\in I}\phi_{\Omega,{n}_{i}}\circ a(n_i)\,d\nu
&= \int_{Y}F\circ a(n_{i_0})\,d\nu=\int_X F\, d\mu+
O\left(e^{-\gamma_1\lfloor n_{i_0} \rfloor}\|F\|_{C^\ell}\right)\\
&=\int_{X}\prod_{i\in I}\phi_{\Omega,{n}_{i}}\circ a(n_i)\,d\mu
+O\left((\ln T)^{K_2-\gamma_1 (B_1-K_1)}\right)\nonumber
\end{align}
since $\mu$ is $a(n)$-invariant.
We choose $B_1$ sufficiently large, so that the second term is negligible.
When $|I|\ge 3$, we pick $J\subset I$ with $|J|=3$.
Applying \eqref{eq:LTT} the H\"older inequality,
we conclude that
\begin{align*}
\int_{X}\prod_{i\in I}\phi_{\Omega,{n}_{i}}\circ a(n_i)\,d\mu
\le L_T^{|I|-3}\int_X \prod_{i\in J}\phi_{\Omega,{n}_{i}}\circ a(n_i)\,d\mu
\le L_T^{|I|-3}\prod_{i\in J} \|\phi_{\Omega,{n}_{i}}\|_{L^3(\mu)}.
\end{align*}
By \eqref{eq:LTT} and Proposition \ref{prop:alphaintoverX}, for any $\delta>0$,
$$
\|\phi_{\Omega,{n}_{i}}\|_{L^3(\mu)}\le L_T^{\delta/3} 
\left(\int_X |\widehat f_{\Omega,{n}_{i}}|^{3-\delta}\,d\mu\right)^{1/3}
\ll_{c,\delta}  L_T^{\delta/3}\cdot b^{(3-\delta)/9}.
$$
Therefore,
\begin{align*}
\int_{X}\prod_{i\in I}\phi_{\Omega,{n}_{i}}\circ a(n_i)\,d\mu
\ll_{c,\delta} L_T^{|I|-3+\delta}\cdot b^{1-\delta/3}.
\end{align*}
We take $B_1$ sufficiently large, so that this bound dominates in \eqref{eq:limm}.
Ultimately, we conclude that 
$$
\int_{Y}\prod_{i\in I}\phi_{\Omega,{n}_{i}}\circ a(n_i)\,d\nu
\ll_{c,\delta} L_T^{|I|-3+\delta}\cdot b^{1-\delta/3}.
$$
The case when $|I|\le 2$ can be handled similarly. Here
one even gets a better bound
$$
\int_{X}\prod_{i\in I}\phi_{\Omega,{n}_{i}}\circ a(n_i)\,d\mu
\ll_c b
$$
from Lemma \ref{l:f_l1_l2}. We conclude that for all partitions $\mc{Q}$ of $\{1,\ldots,r\}$,
$$
\prod_{I\in\mc{Q}}\int_{Y}\prod_{i\in I}\phi_{\Omega,{n}_{i}}\circ a(n_i)\,d\nu
\ll_{c,\delta} L_T^{r-3+r\delta}\cdot b^{1-\delta/3},
$$
which also implies that
$$
\Big|\textup{Cum}\big(\phi_{\Omega,{n}_1}\circ a( n_1),\dotsc,\phi_{\Omega,{n}_r}\circ a( n_r)\big)\Big|
\ll_{c,\delta} L_T^{r-3+r\delta}\cdot b^{1-\delta/3}.
$$
We also note that the number of summands with the additional condition
\eqref{eq:group} is $O_{c,K_1}\big((\ln T)^2(\ln\ln T)^{2r-2}\big)$.

Now we consider the case
\begin{equation}\label{ass:2}
\lfloor n_{1}\rfloor< B_1\,\ln\ln T\qand |n_1|\ge B_2\, \ln\ln T
\end{equation}
for a constant $B_2>B_1$ to be specified later.
Without loss of generality, let us assume that $n_{1,1}\ge n_{1,2}$ since
the other case can be handled similarly. Then
$$
n_{1,2}< B_1\,\ln\ln T\qand n_{1,1}\ge B_2\, \ln\ln T.
$$
In fact, it follows from \eqref{eq:group} that
\begin{equation}\label{ass:2_1}
n_{i,2}< (B_1+K_1)\,\ln\ln T\qand n_{i,1}\ge (B_2-K_1)\, \ln\ln T\quad\hbox{for all $i$.}
\end{equation}
We pick $i_0\in I$ and write
$$
F:={\prod}_{i\in I}\phi_{\Omega,{n}_{i}}\circ a\big(n_i-(n_{i_0,1},0)\big).
$$
Then it follows from \eqref{eq:group} and \eqref{ass:2_1} that
$$
|n_i-(n_{i_0,1},0)|\le (2K_1+B_1)\, \ln\ln T\quad\hbox{for all $i$}.
$$
Therefore, as in the first case,
$\|F\|_{C^\ell}\ll (\ln T)^{K'_2}$
for some $K'_2>0$ depending on $K_0$, $K_1$, and $B_1$. Therefore,
applying Theorem \ref{prop:ED}, we deduce that
\begin{align}\label{eq:limm2}
\int_{Y}\prod_{i\in I}\phi_{\Omega,{n}_{i}}\circ a(n_i)\,d\nu
&= \int_{Y}F\circ a(n_{i_0,1},0)\,d\nu=\int_{Y_1} F\, d\nu_1+
O\left(e^{-\gamma_1 n_{i_0,1}}\|F\|_{C^\ell}\right)\\
&=\int_{Y_1}\prod_{i\in I}\phi_{\Omega,{n}_{i}}\circ a(n_i)\,d\nu_1
+O\left((\ln T)^{K_2'-\gamma_1 (B_2-K_1)}\right)\nonumber
\end{align}
since $\nu_1$ is $a(*,0)$-invariant. We choose $B_2$ sufficiently large,
so that the second term is negligible.
Due to \eqref{eq:LTT} and Proposition \ref{prop:alphaintoverY}, we also have the bound
with any $\delta>0$, 
\begin{align*}
\|\phi_{\Omega,{n}_{i}}\circ a(n_i)\|_{L^2(\nu_1)} &\le L_T^{\delta/2} 
\left(\int_{Y_1} |\widehat f_{\Omega,{n}_{i}}\circ a(n_i)|^{2-\delta}\,d\nu_1\right)^{1/2}
\ll_{c,\delta} L_T^{\delta/2}\cdot b^{(2-\delta)/4}.
\end{align*}
Therefore, as in the previous case, we get
$$
\int_{Y_1}\prod_{i\in I}\phi_{\Omega,{n}_{i}}\circ a(n_i)\,d\nu_1\ll_{c,\delta}
L_T^{|I|-2+\delta}\cdot b^{1-\delta/2}
$$
when $|I|\ge 2$.
Taking $B_2$ sufficiently large, we ensure that this bound dominates  in \eqref{eq:limm2}.
Thus, we derive that 
$$
\int_{Y_1}\prod_{i\in I}\phi_{\Omega,{n}_{i}}\circ a(n_i)\,d\nu_1
\ll_{c,\delta} L_T^{|I|-2+\delta}\cdot b^{1-\delta/2}.
$$
When $|I|=1$, we get a better bound $O_c(b)$ directly from Proposition \ref{prop:L1}.
Therefore, for all partitions $\mc{Q}$ of $\{1,\ldots,r\}$,
$$
\prod_{I\in\mc{Q}}\int_{Y}\prod_{i\in I}\phi_{\Omega,{n}_{i}}\circ a(n_i)\,d\nu
\ll_{c,\delta} L_T^{r-2+r\delta}\cdot b^{1-\delta/2},
$$
and
$$
\Big|\textup{Cum}\big(\phi_{\Omega,{n}_1}\circ a( n_1),\dotsc,\phi_{\Omega,{n}_r}\circ a( n_r)\big)\Big|
\ll_{c,\delta} L_T^{r-2+r\delta}\cdot b^{1-\delta/2}.
$$
We also note that the number of summands with the additional conditions \eqref{eq:group}
and \eqref{ass:2} is $O_{B_1,K_1}\big((\ln T)(\ln\ln T)^{2r-1}\big)$.

Finally, we consider the case
\begin{equation}\label{ass:3}
|n_1|< B_2\, \ln\ln T.
\end{equation}
For a non-empty subset $I$ of $\{1,\ldots,r\}$ and $i_0\in I$,
it follows from Proposition \ref{prop:L1},
\begin{align*}
\int_{Y}\prod_{i\in I}\phi_{\Omega,{n}_{i}}\circ a(n_i)\,d\nu
\le L_T^{|I|-1}\int_X \phi_{\Omega,{n}_{i_0}}\circ a(n_{i_0})\,d\nu\ll_c
L_T^{|I|-1}\cdot b.
\end{align*}
Then for all partitions $\mc{Q}$ of $\{1,\ldots,r\}$,
$$
\prod_{I\in\mc{Q}}\int_{Y}\prod_{i\in I}\phi_{\Omega,{n}_{i}}\circ a(n_i)\,d\nu
\ll_c L_T^{r-1}\cdot b,
$$
and
$$
\Big|\textup{Cum}\big(\phi_{\Omega,{n}_1}\circ a( n_1),\dotsc,\phi_{\Omega,{n}_r}\circ a( n_r)\big)\Big|
\ll_c L_T^{r-1}\cdot b.
$$
Here the number of summands satisfying \eqref{eq:group} and 
\eqref{ass:3} is $O_{B_2,K_1}\big((\ln\ln T)^{2r}\big)$.

Now combine the estimates obtained in the above three cases to get for all $\delta>0$,
\begin{align*}
&\sum_{({n}_{1},\dotsc,{n}_{r})\in \cF_\Omega^{r}\cap\Delta(\beta_r)}
\Big|\textup{Cum}\big(\varphi_{\Omega,{n}_1}\circ a( n_1),\dotsc,\varphi_{\Omega,{n}_r}\circ a( n_r)\big)\Big| \\
\ll_{c,\delta}  &\;
L_T^{r-3+r\delta} b^{1-\delta/3} (\ln T)^2(\ln\ln T)^{2r-2}\\
&+
L_T^{r-2+r\delta} b^{1-\delta/2} (\ln T)(\ln\ln T)^{2r-1}
+
L_T^{r-1} b (\ln\ln T)^{2r}.
\end{align*}
Using our assumptions on $L_T$ and $b$,
we check by direct computation that
taking $\delta>0$ sufficiently small, the above bound is $o(V_\Omega^{r/2})$ as $T\to \infty$.
\end{proof}

\begin{lem}
\label{lem:sparse}
Let $M,M'\ge 0$ and $\mc{Q}$ be a non-trivial partition 
of $\{1,\ldots,r\}$.
Then under the assumptions \eqref{eq:high_cond}, for all $(n_{1},\dotsc,n_{r})\in \Delta_{\mc{Q}}\big(M\, \ln\ln T,M'\, \ln\ln T\big)$,
$$
\Big|\textup{Cum}(\varphi_{\Omega,{n}_1}\circ a(n_1),\dotsc,\varphi_{\Omega,{n}_r}\circ a(n_r))\Big|\ll
(\ln T)^{r (K_0\sigma_\ell+1)+ r\theta_\ell M-\eta_rM'}.
$$
\end{lem}

\begin{proof}
For bounded measurable functions $\phi_1,\ldots ,\phi_r:X\to \bR$,
one defines the conditional cumulant with respect to a partition $\mc{Q}$ as
$$
\textup{Cum}(\phi_{1},\dotsc,\phi_{r}|\mc{Q}):=\sum_{\mathcal{P}}(-1)^{|\mathcal{{P}}|-1}(|\mc{P}|-1)!\prod_{I\in\mc{P}}\prod_{J\in\mc{Q}} \int_{X}{\prod}_{i\in I\cap J}\phi_{i}\,d\nu.
$$
Here the product over an empty set is understood to be $1$.
It is well-known (see, for instance, \cite[Proposition~8.1]{BGmult}) that
$$
\textup{Cum}(\phi_{1},\dotsc,\phi_{r}|\mc{Q})=0
$$
for every non-trivial partition $\mc{Q}$.
Therefore, it will be sufficient to check that 
$$
\textup{Cum}(\varphi_{\Omega,{n}_1}\circ a(n_1),\dotsc,\varphi_{\Omega,{n}_r}\circ a(n_r))\approx \textup{Cum}(\varphi_{\Omega,{n}_1}\circ a(n_1),\dotsc,\varphi_{\Omega,{n}_r}\circ a(n_r) | \mc{Q})
$$
with the above error term. In fact, we show that
for any partition $\cP$,
$$
\prod_{I\in\mc{P}}\int_{X}{\prod}_{i\in I}\varphi_{\Omega,n_i}\circ a(n_i)\,d\nu
\approx \prod_{I\in\mc{P}}\prod_{J\in\mc{Q}} \int_{X}{\prod}_{i\in I\cap J}\varphi_{\Omega,n_i}\circ a(n_i)\,d\nu.
$$
When $I\cap J\ne \emptyset$, we pick $i_J\in I\cap J$ and set 
$$
F_{I,J}:={\prod}_{i\in I\cap J}\varphi_{\Omega,{n}_{i}}\circ a(n_i- n_{i_{J}}).
$$
Since $(n_{1},\dotsc,n_{r})\in \Delta_{\mc{Q}}\big(M\, \ln\ln T,M'\, \ln\ln T\big)$,
$$
|n_i- n_{i_{J}}|\le M\ln\ln T,
$$
so that it follows from \eqref{eq:norm_bound} that for some $\theta_\ell>0$,
\begin{align*}
\|F_{I,J}\|_{C^\ell}&\ll_\ell{\prod}_{i\in I\cap J}\|\varphi_{\Omega,{n}_{i}}\circ a(n_i- n_{i_{J}})\|_{C^\ell}
\ll_\ell {\prod}_{i\in I\cap J}\|\varphi_{\Omega,{n}_{i}}\|_{C^\ell}
e^{\theta_\ell |n_i- n_{i_{J}}|}\\
&\ll_\ell  (\varepsilon_T^{-\sigma_\ell}L_T e^{\theta_\ell M\ln\ln T})^{|I\cap J|}
\le (\ln T)^{|I\cap J|(K_0\sigma_\ell+\theta_\ell M)} L_T^{|I\cap J|}.
\end{align*}
Applying Theorem \ref{prop:integrability}, we deduce that 
\begin{align}
\int_Y {\prod}_{i\in I}\varphi_{\Omega,n_i}\circ a(n_i)\, d\nu =&
\int_Y {\prod}_{J\in\mc{Q}} F_{I,J}\circ a(n_{i_J})  d\nu \nonumber \\
 = &{\prod}_{J\in\mc{Q}} \int_Y F_{I,J}\circ a(n_{i_J})  d\nu \nonumber \\
&\quad\quad+O_r\left (e^{-\eta_r \min_{J_1\ne J_2} |n_{i_{J_1}}-n_{i_{J_2}}|}  {\prod}_{J\in\mc{Q}} \|F_{I,J}\|_{C^\ell} \right) \nonumber\\
=& \prod_{J\in\mc{Q}} \int_{Y}{\prod}_{i\in I\cap J}\varphi_{\Omega,n_i}\circ a(n_i)\,d\nu \nonumber\\
&\quad\quad + O_{r} \Big((\ln T)^{|I|(K_0\sigma_\ell+\theta_\ell M)-\eta_rM'}L_T^{|I|} \Big).
\label{eq:corr1}
\end{align}
For subsets $I\subset \{1,\ldots,r\}$, we also have the bound
\begin{align}\label{eq:corr2}
\left|\int_Y {\prod}_{i\in I}\varphi_{\Omega,n_i}\circ a(n_i)\, d\nu\right|
\le L_T^{|I|},
\end{align}
which follows from \eqref{eq:phi_uniform}.
Since $L_T\le V_\Omega^{1/2}\ll \ln T$, 
combining \eqref{eq:corr1} and \eqref{eq:corr2}, we conclude that 
\begin{align*}
\prod_{I\in\mc{P}}\int_{X}{\prod}_{i\in I}\varphi_{\Omega,n_i}\circ a(n_i)\,d\nu
= & \prod_{I\in\mc{P}}\prod_{J\in\mc{Q}} \int_{X}{\prod}_{i\in I\cap J}\varphi_{\Omega,n_i}\circ a(n_i)\,d\nu\\
&\quad\quad\quad + O_{r} \Big((\ln T)^{r (K_0\sigma_\ell+1)+ r\theta_\ell M-\eta_rM'} \Big).
\end{align*}
This implies the  lemma.
\end{proof}

Now we complete the proof of Proposition \ref{prop:higherlimit} using the decomposition \eqref{eq:decomp0}. Let us pick constants 
$$
0=M_0<M_1<3M_1<M_2<\cdots<M_{r-1}<3 M_{r-1}<M_r
$$
and $\beta_i=M_i\, \ln\ln T$.
It follows from Lemma \ref{lem:sparse} that 
\begin{align*}
&\left|\sum_{({n}_{1},\dotsc,{n}_{r})\in \cF_\Omega^r\cap \Delta_{\mc{Q}}(3\beta_j,\beta_{j+1})}\textup{Cum}\big(\varphi_{\Omega,{n}_1}\circ a(n_1),\dotsc,\varphi_{\Omega,{n}_r}\circ a(n_r)\big)\right| \\
\ll_r & |\cF_\Omega|^r (\ln T)^{r (K_0\sigma_\ell+1)+ r\theta_\ell M_i-\eta_rM_{i+1}} \ll_c (\ln T)^{r (K_0\sigma_\ell+3)+ r\theta_\ell M_i-\eta_rM_{i+1}}.
\end{align*}
We pick the parameters $M_i$ recursively, so that 
$r (K_0\sigma_\ell+3)+ r\theta_\ell M_i-\eta_r M_{i+1}<0$ and $M_r=O_{K_0,\ell}(1)$.
Then we get
$$
\sum_{({n}_{1},\dotsc,{n}_{r})\in \cF_\Omega^r\cap \Delta_{\mc{Q}}(3\beta_j,\beta_{j+1})}\textup{Cum}\big(\varphi_{\Omega,{n}_1}\circ a(n_1),\dotsc,\varphi_{\Omega,{n}_r}\circ a(n_r)\big)
= o(1)
$$
as $T\to\infty$. Also by Lemma \ref{lem:cluster}
\begin{align*}
\sum_{({n}_{1},\dotsc,{n}_{r})\in \cF_T^{r}\cap\Delta(\beta_r)}
\textup{Cum}_{r}\big(\varphi_{\Omega,{n}_1}\circ a( n_1),\dotsc,\varphi_{\Omega,{n}_r}\circ a( n_r)\big)
=o(V_\Omega^{r/2})
\end{align*}
as $T\to\infty$, which completes the proof.

\section{Completion of proof} \label{sec:compofp}

Now we finish the proof of Theorem \ref{th:main2}.
We combine Proposition \ref{lem:convinL1}, Proposition \ref{th:variance}, and Proposition \ref{prop:higherlimit}.
Since
$$
\int_Y \Phi_\Omega\, d\nu=0, \quad
V_\Omega^{-1}\int_{Y}\Phi_\Omega^{2}\, d\nu \to 
\sigma^2>0,\quad 
\textup{Cum}_{r}\left(\Phi_\Omega\right)=o(V_\Omega^{r/2})\quad\hbox{for $r\ge 3$},
$$
it follows that the random variable $V_\Omega^{-1/2}\Phi_\Omega$ converges as $T\to\infty$ in distribution to the Normal Law
with variance $\sigma^2$. Furthermore, since
$$
\big\|\Psi_{\Omega}-\Phi_\Omega\big\|_{{L}^1(\nu)}=o\left(V_\Omega^{1/2}\right),
$$
the random variable $V_\Omega^{-1/2}\Psi_\Omega$ has the same limiting distribution.
In view of \eqref{eq:connn}, this implies the theorem.

It remains to verify that the assumptions of these propositions can be simultaneously satisfied, namely, the conditions \eqref{eq:cond_approx},
\eqref{eq:var_cond}, \eqref{eq:var_cond333}, \eqref{eq:high_cond}.
We take $L_T= V_{\Omega}^{1/2}$ and $\eps_T = (\ln T)^{-K_0}$ with $K_0>0$.
Then we are required to have 
\begin{align*}
(\ln T)^{-(K_0-K_1)}\le a \leq b (\ln \ln T)^{-4-\varepsilon},\;\;
b^{-1/3}(\ln T)^{\theta}\le V_\Omega^{1/2},\;\; b\ge (\ln T)^{-2\min(1,K_1-1)+\theta}
\end{align*}
for some $0<K_1<K_0$ and $\theta,\varepsilon>0$.
We recall (cf. \eqref{eq;voll}) that $V_\Omega\sim (\ln T)^2b$ as $T\to\infty$, so that
these conditions can be arranged when
$$
(\ln T)^{-K}\le a \leq b (\ln \ln T)^{-4-\varepsilon}\qand
b\ge (\ln T)^{-6/5+\theta}
$$
for some $K,\varepsilon,\theta>0$. This gives the main theorem under the additional assumption that 
$$
a\ge (\ln T)^{-K}\quad\hbox{for some $K>0$.}
$$

Finally, we show that  solutions with small products give negligible contributions.
For this purpose, we use the following lemma:

\begin{lem}\label{l:small}
Let 
$$
\Theta_T:=
\left\{(x_1,x_2,y)\in\mb{R}^{3}:
\; |x_1x_2|\, y\le(\ln T)^{-K},\;
  |x_1|,\,|x_2|\leq c,\;
  1\le y< T
\right\}.
$$
Then when $K>2$, for a.e. $\ul x\in [0,1)^2$,
$$
\left|\bigcup_{T> 1} \Theta_T\cap \Lambda_{\ul x}\right|<\infty.
$$
\end{lem}

\begin{proof}
The set in question consists of vectors $(qx_1-p_1,qx_2-p_2,q)$ with $(p_1,p_2,q)\in\bZ^3$ such that 
$$
|qx_1-p_1|\cdot |qx_2-p_2|\cdot q\le (\ln T)^{-K},\;\;\;
  |qx_1-p_1|, |qx_2-p_2|\leq c,\;\;\;
  1\le q< T
$$
for some $T>1$. If this set is infinite, it follows the inequalities
$$
|qx_1-p_1|\cdot |qx_2-p_2|\le(\ln q)^{-K}q^{-1},\;\;\;
  |qx_1-p_1|, |qx_2-p_2|\leq c
$$
have integral solutions for infinitely many $q\in\bN$.
In other words,
$$
\left|\bigcup_{T\ge 1} \Theta_T\cap \Lambda_{\ul x}\right|=\infty
\quad\Longrightarrow \quad\ul x\in \limsup A_q,
$$
where
$$
A_q:=\left\{\ul x\in [0,1)^2:\, \exists p_1,p_2\in \bZ:  (qx_1-p_1,qx_2-p_2)\in \Upsilon\left((\ln q)^{-K}q^{-1}\right)\right\}
$$
and
$$
\Upsilon(\varepsilon):=\left\{(x_1,x_2) \in\mb{R}^{2}:
\; |x_1x_2|\le\varepsilon,\;
  |x_1|,|x_2|\leq c
\right\}.
$$
One checks that 
$$
\Vol(\Upsilon(\varepsilon))\ll_c \varepsilon\ln (1/\varepsilon).
$$
for sufficiently small $\varepsilon$, and consequently
$$
\Vol(A_q)\ll_c (\ln q)^{1-K}q^{-1}.
$$
Hence, the claim follows from the Borel--Cantelli lemma.
\end{proof}

Now we are ready to complete the proof of the main theorem.
Let us write 
$$
\Omega=\Omega'\sqcup \Omega'',
$$
where the set
$\Omega'$ consists of the points of $\Omega$ satisfying 
additionally $|x_1x_2|\, y>(\ln T)^{-K}$ with fixed sufficiently large $K$,
and $\Omega''$ is its complement. Then
$$
|\Omega\cap \Lambda_{\ul x}|-\Vol(\Omega)=
\Big(|\Omega'\cap \Lambda_{\ul x}|-\Vol(\Omega')\Big)
+
\Big(|\Omega''\cap \Lambda_{\ul x}|-\Vol(\Omega'')\Big).
$$
We have already established that 
$\Vol(\Omega')^{-1/2}\big(|\Omega'\cap \Lambda_{\ul x}|-\Vol(\Omega')\big)$, $\ul x\in [0,1)^2$, 
converges in distribution to the Normal Law,
because $\Omega'$ can be viewed as the set \eqref{eq:omega0} with
the lower bound $a':=\max(a,(\ln T)^{-K})$.
Further, we note that
$$
\Vol(\Omega'')\le \int_1^T \Vol\left(\Upsilon\big(y^{-1}(\ln T)^{-K}\big)\right)\,dy\ll_K (\ln T)^{1-K}\ln\ln T=o\big(\Vol(\Omega)^{1/2}\big)
$$
as $T\to \infty$. Therefore, 
$\Vol(\Omega)^{-1/2}\big(|\Omega'\cap \Lambda_{\ul x}|-\Vol(\Omega')\big)$
also converges in distribution to the same limit. 
Further, it follows from Lemma \ref{l:small} that
$\Vol(\Omega)^{-1/2}\big(|\Omega''\cap \Lambda_{\ul x}|-\Vol(\Omega'')\big) \to 0$ for a.e. $\ul x\in [0,1)^2$,
and consequently this quantity also converges to zero in measure.
Therefore, $\Vol(\Omega)^{-1/2}\big(|\Omega\cap \Lambda_{\ul x}|-\Vol(\Omega)\big)$
converges to the same Normal Law.


\end{document}